\documentclass[a4paper,10pt,draft]{article}
\usepackage{amsfonts}
\usepackage{amsmath}
\usepackage{amssymb}
\usepackage{amsthm}
\usepackage[american]{babel}
\usepackage[T1]{fontenc}
\usepackage[usenames]{color}
\definecolor{red}{rgb}{1.0,0.0,0.0}
\def\red#1{{\textcolor{red}{#1}}}
\definecolor{blu}{rgb}{0.0,0.0,1.0}

\definecolor{gre}{rgb}{0.03,0.50,0.03}

\newtheorem{Theorem}{Theorem}[section]
\newtheorem{Definition}[Theorem]{Definition}
\newtheorem{Proposition}[Theorem]{Proposition}
\newtheorem{Lemma}[Theorem]{Lemma}
\newtheorem{Notation}[Theorem]{Notation}
\newtheorem{Corollary}[Theorem]{Corollary}
\newtheorem{Remark}[Theorem]{Remark}

\newtheorem{Hypothesis}[Theorem]{Hypothesis}

\newcommand{\nd}{\stackrel{\textrm{def}}{=}}
\newcommand{\N}{{\mathbb N}}
\newcommand{\R}{{\mathbb R}}
\newcommand{\ud}{\mathrm{d}}
\newcommand{\eps}{\varepsilon}
\newcommand{\Lc}{{\cal L}}

\newcommand{\bey}{\begin{eqnarray*}}
\newcommand{\eey}{\end{eqnarray*}}
\newcommand{\beq}{\begin{equation}}
\newcommand{\eeq}{\end{equation}}

\def\Swiech
{{\accent"13S}wie{\hbox{\kern -0.21em\lower
0.79ex\hbox{$\textfont1=\scriptfont1
\lhook$}}}ch}

\def\qedo{\hbox{\hskip 6pt\vrule width6pt height7pt
depth1pt  \hskip1pt}\bigskip}

\def\<{\langle }
\def\>{\rangle }

\def\R{\mathbb R}
\def\N{\mathbb N}

\def\cald{{\cal D}}

\def\call{{\cal L}}

\def\calq{{\cal Q}}

\def\calu{{\cal U}}

\def\call{{\cal L}}

\def\calr{{\cal R}}

\def\1{\mathbf 1}
\def\to{\rightarrow}

\author{
\vspace{-3.5truecm}
P. Acquistapace\footnote{Dipartimento di Matematica,
Universit\`a di Pisa, e-mail: paolo.acquistapace@unipi.it},
\; F.Gozzi\footnote {Dipartimento di Economia e Finanza,
Universit\`a \emph{LUISS - Guido Carli} Roma;
e-mail: fgozzi@luiss.it}}
\title{Minimum energy
with infinite horizon:
\\ from stationary  to non-stationary states}
\vspace{-1truecm}
\date{}
\begin{document}
\maketitle
\vspace{-1.0truecm}

\begin{abstract}
We study a non standard infinite horizon, infinite dimensional
linear-quadratic control problem arising in the physics of non-stationary states (see e.g. \cite{BDGJL4,BertiniGabrielliLebowitz05}): finding the minimum energy to drive a given stationary state $\bar x=0$ (at time $t=-\infty$) into an arbitrary non-stationary state $x$ (at time $t=0$). This is the opposite to what is commonly studied in the literature on null controllability (where one drives a generic state $x$ into the equilibrium state $\bar x=0$). Consequently, the Algebraic Riccati Equation (ARE) associated to this problem is non-standard since the sign of the linear part is opposite to the usual one and since it is intrinsically unbounded. Hence the standard theory of AREs does not apply. The analogous finite horizon problem has been studied in the companion paper \cite{AcquistapaceGozzi17}. Here, similarly to such paper, we prove that the linear selfadjoint operator associated to the value function is a solution of the above mentioned ARE. Moreover, differently to \cite{AcquistapaceGozzi17}, we prove that such solution is the maximal one. The first main result (Theorem \ref{th:maximalARE}) is proved by approximating the problem with suitable auxiliary finite horizon problems (which are different from the one studied in \cite{AcquistapaceGozzi17}). Finally in the special case where the involved operators commute we characterize all solutions of the ARE (Theorem \ref{th:sol=proj}) and we apply this to the Landau-Ginzburg model.

\end{abstract}




\tableofcontents

\section{Introduction}\label{INTRO}

We study a non standard infinite dimensional,
infinite horizon,
linear-quadratic control problem:
finding the minimum energy to drive a given stationary state $\bar x=0$
(at time $t=-\infty$) into an arbitrary non-stationary state $x$ (at time $t=0$).

This kind of problems arises in the control representation of the rate function for a class of large deviation problems (see e.g. \cite{DaPratoPritchardZabczyk91} and the references quoted therein; see also \cite[Chapter 8]{FengKurtzbook} for an introduction to the subject).  It is motivated by applications in the physics of non-equilibrium states and in this context it has been studied in various papers, see e.g. \cite{BDGJL1,BDGJL2,BDGJL3,BDGJL4,BDGJL5,BertiniGabrielliLebowitz05}
(see Section \ref{PHYSICS} for a description of a model case).

The main goal here, as a departure point of the theory, is to apply the dynamic programming approach to characterize the value function as the unique (or maximal/minimal) solution of the associated Hamilton-Jacobi-Bellman (HJB) equation, a problem left open e.g. in \cite{BDGJL4,BertiniGabrielliLebowitz05}.
This problem is quite difficult since it deals with the opposite to what is commonly studied in the literature on null controllability
(where one drives a generic state $x$ into the equilibrium state $\bar x=0$).
For this reason we start studying here the simplest case, i.e. when the state equation is linear and the energy functional is purely quadratic:
so the problem falls into the class of linear-quadratic optimal control problems,
the value function is quadratic, and the associated HJB equation
reduces to an Algebraic Riccati Equation (ARE).

The above feature (i.e. the fact we bring $0$ to $x$ instead of the opposite)
implies that the ARE associated to this problem is non-standard for two main reasons:
first, the sign of the linear part is opposite to the usual one;
second, since the set of reachable $x$ is strictly smaller than the whole state space $X$,
the solution is intrinsically unbounded in $X$.
The combination of these two difficulties does not allow to apply
the standard theory of AREs.

In the companion paper \cite{AcquistapaceGozzi17} we studied, as a first step, the associated finite horizon case. Here we partially exploit the results of such paper to deal with the more interesting infinite horizon case, which is the one that arises in the above mentioned papers
in physics.

Our main results (Theorems \ref{th:maximalARE} and \ref{th:sol=proj}) show that, under a null controllability assumption (after a given time $T_0\ge 0$) and a coercivity assumption on the control operator,
the linear selfadjoint operator $P$ associated to the value function is the maximal solution of the above mentioned ARE. The first result concerns the general case with some restrictions on the class of solutions, while the second one looks at the case where the state and the control operators commute, without any restriction on the class of solutions.

This is only partially similar to what has been done in \cite{AcquistapaceGozzi17}.
Indeed, the proof that $P$ is a solution of ARE is substantially similar to what is done
in \cite[Section 4.3]{AcquistapaceGozzi17}.
On the other hand, while in \cite[Section 4.4]{AcquistapaceGozzi17}
we prove a partial uniqueness result (i.e. uniqueness in a suitable family of invertible operators), here we are able to prove, through a delicate comparison argument (based on a nontrivial approximation procedure), that $P$ is the maximal solution of the associated ARE.

To prove the comparison argument (which is the content of the key Lemma \ref{lem:massimalitaN})
we need to introduce a family of
auxiliary finite horizon problems, which are different from the one studied in \cite{AcquistapaceGozzi17}.

Finally, in the special case where the involved operators commute, we are able, again differently from the finite horizon case, to characterize all solutions of the ARE. This allows to apply our result to the case of Landau-Ginzburg model.

\subsection{Plan of the paper}

In Section 2 we illustrate the problem and the strategy
to show the main results.
It is divided in three subsections: in the first we present the state equation and the main Hypothesis;
in the second one we describe our minimum energy problem;
the third subsection briefly explains the method used to prove our main results. Section 3 concerns the study of the auxiliary problem. After devoting the first part of the section to some basic results on it,
we show, in Subsection \ref{sub:MaxPN}, the comparison Lemma \ref{lem:massimalitaN} which will be used to prove the maximality result in the infinite horizon case. Section 4 is devoted to the main problem and the main maximality result. In Section 5 we analyze the case when the operators $A$ and $BB^*$ commute. In Section 6 we present, as an example, a special case of the motivating problem given in \cite{BDGJL4} (the case of the so-called Landau-Ginzburg model): we show that it falls into the class of problems treated in this paper.

%
%
%

\section{The problem and the main results}


\subsection{The state equation}\label{SE:STATEEQUATION}

\begin{Notation}\label{eq:call}
Given any two Banach spaces $Y$ and $Z$, we denote by $\call(Y,Z)$ the set of all linear bounded operators from $Y$ to $Z$, writing $\call(Y)$ when $Z=Y$.
When $Y$ is a Hilbert space we denote by $\call_+(Y)$ the set of all elements of $\call(Y)$ which are selfadjoint and nonnegative.
\end{Notation}


Let $-\infty<s<t<+\infty$. Consider the abstract linear equation
\begin{equation}
\label{eq:state-fin-new} \left\{
\begin{array}{l}
y'(r)=Ay(r)+Bu(r), \quad r\in \, ]s,t], \\[2mm]
y(s) = x \in X,
\end{array}
\right.
\end{equation}
under the following assumption.

\begin{Hypothesis}\label{hp:main} \begin{description}
\item[]
\item[(i)] $X$, the state space, and $U$, the control space, are real separable Hilbert spaces;
\item[(ii)] $A:{\cal D}(A)\subseteq X \to X$ is the generator of a $C_0$-semigroup on $X$ such that
\begin{equation}
\label{eq:AtipoMomega} \| e^{t A}\|_{{\cal L}(X)} \leq M e^{-\omega t},
\qquad t\ge 0,
\end{equation}
for given constants $M>0$ and $\omega>0$;
\item[(iii)] $B:U \to X$ is a bounded linear operator;
\item[(iv)] $u$, the control strategy, belongs to $L^2(s,t;U)$.
\end{description}
\end{Hypothesis}

We recall the following well known result,
pointed out e.g. in \cite[Proposition 2.2]{AcquistapaceGozzi17}.

\begin{Proposition}
\label{prop:solcontinua} For $-\infty<s<t<+\infty$, $x\in X$ and
$u\in L^2(s,t;U)$, the mild solution of (\ref{eq:state-fin-new}),
defined by
\begin{equation}\label{vardcost}
y(r;s,x,u) = e^{(r-s)A}x + \int_s^r e^{(r-\sigma)A}Bu(\sigma) \, \ud \sigma, \quad
r\in [s,t],
\end{equation}
is in $C([s,t],X)$.
\end{Proposition}

We now consider the state equation in the half-line $\,]-\infty, t]$:
\begin{equation}
\label{eq:state-inf-new} \left\{
\begin{array}{l}
y'(r)=Ay(r)+Bu(r), \quad r\in\,]-\infty,t], \\[2mm]
\displaystyle \lim_{s \to - \infty}y(s) = 0.
\end{array}
\right.
\end{equation}
Since (\ref{eq:state-inf-new}) is not completely standard we
introduce the following definition of solution.
\begin{Definition}
\label{def:delsolstate-new} Given $u \in L^{2}(-\infty ,t;U)$, we say that
$y\in C(\, ]-\infty ,t];X)$
is a solution of (\ref{eq:state-inf-new}) if for every $-\infty <
r_1 \leq r_2 \leq t$ we have
\begin{equation}
\label{eq:defsolution-new} y(r_2) =  e^{(r_2 - r_1)A} y(r_1)+
\int_{r_1}^{r_2} e^{(r_2-\tau)A} Bu(\tau) \ud \tau.
\end{equation}
and
\begin{equation}
\label{eq:defsol-dato-new} \lim_{s \to - \infty }y(s)= 0.
\end{equation}
\end{Definition}

\begin{Lemma}
\label{lm:existencesol-new} Given any $u \in L^{2}( -\infty,t;U)$, there exists a unique solution of the
Cauchy problem
(\ref{eq:state-inf-new}) and it is given by
\begin{equation}
\label{eq:mild-inf-new} y(r;-\infty,0,u):=\int_{-\infty
}^{r}e^{(r-\tau)A}B u(\tau)\, \ud \tau, \qquad r \le t.
\end{equation}
\end{Lemma}

\begin{proof}
We prove first that the function $y(\cdot;-\infty,0,u)$ given by
(\ref{eq:mild-inf-new}) is continuous. Fixed $r_1<r_2\leq t$, we have
\begin{eqnarray*}
\lefteqn{y(r_2;-\infty,0,u)-y(r_1,-\infty,0,u) =}\\
& & = \int_{-\infty }^{r_2}e^{(r_2-\tau)A}B u(\tau)\, \ud \tau- \int_{-\infty
}^{r_1}e^{(r_1-\tau)A}B u(\tau)\, \ud \tau = \\
& & = \int_{-\infty}^{r_1}\left(e^{(r_2-r_1)A}-I\right) e^{(r_1-\tau)A} B u(\tau)\, \ud \tau +\int_{r_1}^{r_2} e^{(r_2-\tau)A}Bu(\tau)\, \ud \tau,
\end{eqnarray*}
and then continuity follows by standard
arguments. We now prove that (\ref{eq:defsolution-new}) holds. For $-\infty
< r_1\leq r_2 \leq t$, we have
\begin{eqnarray*}
y(r_2;-\infty,0,u) & = & \int_{-\infty }^{r_2}e^{(r_2-\tau) A}Bu(\tau)\, \ud \tau = \\
& = & e^{(r_2-r_1)A} \int_{-\infty}^{r_1} e^{(r_1-\tau)A} Bu(\tau)\, \ud
\tau + \int_{r_1}^{r_2} e^{(r_2-\tau)A}B u(\tau)\, \ud \tau=\\
& = & e^{(r_2-r_1)A} y(r_1;-\infty,0,u) + \int_{r_1}^{r_2}
e^{(r_2-\tau)A} B u(\tau)\, \ud \tau,
\end{eqnarray*}
so (\ref{eq:defsolution-new}) is satisfied. Moreover letting $r\to
-\infty$, since $u\in L^2(-\infty,t;U)$ and thanks to equation
(\ref{eq:AtipoMomega}), we have $y(r;-\infty,x,u)
\to 0$ as ${r\to-\infty}$.

In order to prove uniqueness, consider
two solutions $y_1(\cdot)$ and $y_2(\cdot)$ and a point $r\in
(-\infty,t)$. Since $y_1(\cdot)$ and $y_2(\cdot)$ satisfy
(\ref{eq:defsolution-new}), for their difference we have, for
$r_0<r<t$,
$$\|y_1(r) - y_2(r)\|_{X} = \|e^{(r-r_0)A}(y_1(r_0) - y_2(r_0))\|_X \leq
M\,e^{-(r-r_0)\omega} \|(y_1(r_0) - y_2(r_0))\|_X\,.
$$
As $y_1(\cdot)$ and $y_2(\cdot)$ satisfy
(\ref{eq:defsol-dato-new}), letting $r_0\to - \infty$ above we get $y_1(r)=y_2(r)$ for every $r<t$.
\end{proof}

\begin{Remark}
{\em
Notice that, if the initial condition \eqref{eq:defsol-dato-new} is not zero, then the above equation cannot have any solution. Indeed any solution $y(\cdot;-\infty,x,u)$ of the state equation (\ref{eq:state-inf-new}), with $0$ replaced by $x \in X\setminus\{0\}$ in \eqref{eq:defsol-dato-new}, must satisfy (\ref{eq:defsolution-new})
and $\lim_{s \to - \infty}y(s)=x$. But, as $r_1 \to - \infty$, (\ref{eq:defsolution-new})
implies, as in (\ref{eq:mild-inf-new}), that
\begin{equation}
y(r_2;-\infty,x,u):=\int_{-\infty
}^{r_2}e^{(r_2-\tau)A}B u(\tau) \ud \tau, \qquad r_2 \le t.
\end{equation}
Taking the limit as $r_2 \to - \infty$ we get $x=0$, a contradiction.
}
\hfill\qedo
\end{Remark}



\subsection{Minimum energy problems with infinite horizon
and associated Riccati equation}
\label{SSE:introinfhor}

To better clarify our results we state, roughly and informally, the mathematical problem (see Section \ref{sub:infhor} for a precise description).
The state space $X$ and the control space $U$ are both real separable Hilbert spaces.
We take the linear controlled system in $X$
\begin{equation}\label{storinf}
\left\{ \begin{array}{ll} y'(s)=Ay(s)+Bu(s), \quad s\in \,]-\infty,0], \\[2mm]
y(-\infty) = 0, \end{array}\right.
\end{equation}
where $A: {\cal D}(A) \subset X \to X$ generates a strongly continuous semigroup
and $B:U\to X$ is a linear, possibly unbounded operator.
Given a point $x \in X$ we consider the set ${\cal U}_{[-\infty,0]}(0,x) $ of all
square integrable control strategies that drive the system from the equilibrium state $0$
(at time $t=-\infty$) into the generic non-equilibrium state $x$ (at time
$t=0$). It is well known (see Proposition \ref{pr:optcoupleinfty}) that the set
${\cal U}_{[-\infty,0]}(0,x) $ is nonempty if and only if $x\in H$, where $H$
is a suitable subspace of $X$ that can be endowed with its own Hilbert structure
(see next subsection for the precise definition of $H$ and Subsection \ref{SSE:SPACEH} for its properties).

We want to minimize the ``energy-like'' functional
\begin{equation}\label{fzorinf}
J_{[-\infty,0]}(u) =\frac12 \int_{-\infty}^0 \|u(s)\|^2_U\, \ud s.
\end{equation}
As usual the value function $V_\infty$ is defined as
\begin{equation}\label{eq:VFinfhornew}
V_\infty(x) = \inf_{u\in {\cal U}_{[-\infty,0]}(0,x)} J_{[-\infty,0]}(u),
\end{equation}
and it is finite only when $x \in H$.

The peculiarity of the problem with respect to the most studied
minimum energy problems in Hilbert spaces
(see e.g. \cite{Carja93} \cite{DaPratoPritchardZabczyk91},
\cite{Emir89,Emir95}, \cite{GozziLoreti99}, \cite{PriolaZabczyk03},
and the general surveys \cite{BDDM07}, \cite{CurtainPritchardbook},
\cite{LTbook1,LTbook2}, \cite{Zabczyk92})
is that it gives rise to an Algebraic
Riccati Equation with a `wrong' sign in the linear term to which, to our knowledge, the
standard theory developed in the current literature does not apply.

Indeed, the associated ARE in $X$ (with unknown $R$), which can be found applying the dynamic programming principle, is, formally,
\begin{equation}\label{eq:algriccintroX}
0=-\langle Ax, Ry\rangle_X -\langle Rx, Ay\rangle_X- \langle
B^*R x, B^*Ry \rangle_U\,, \quad x,y \in {\cal D}(A)\cap {\cal D}(R) .
\end{equation}
Since $R$ is unbounded (this comes from the fact that $V_\infty$ is defined only in $H$), it is convenient to rewrite (\ref{eq:algriccintroX}) in $H$ (with unknown $P$, which is now a bounded operator on $H$). This way we get the equation
\begin{equation}\label{eq:algriccintroH}
0=-\langle Ax, Py\rangle_H -\langle Px, Ay\rangle_H- \langle
B^*Q_\infty^{-1}P x,B^*Q_\infty^{-1}Py \rangle_U,
\end{equation}
or, transforming the inner products in $H$ into inner products in $X$,
\begin{equation}\label{eq:algriccintro}
0=-\langle Ax,Q_\infty^{-1} Py\rangle_X -\langle Q_\infty^{-1}Px, Ay\rangle_X- \langle
B^*Q_\infty^{-1}P x,B^*Q_\infty^{-1}Py \rangle_U,
\end{equation}
In the last two equations $Q_\infty$ is the so-called controllability operator (see \eqref{eq:contropintro}) and $Q_\infty^{-1}$ denotes its pseudoinverse, which is, in general, unbounded. Moreover the last two equations make sense for $x,y$ belonging to suitable sets to be specified later on.
For more details on how these equations arise, the definitions of solution, and the relations among them, see the discussion at the beginning of Subsection \ref{SSE:ARE}. Here we just observe that the form of equation \eqref{eq:algriccintro} turns out to be more suitable to prove our main results.


The `wrong' sign in the linear term\footnote{Evidently the two terms in equation \eqref{eq:algriccintro} (or \eqref{eq:algriccintroX}, or \eqref{eq:algriccintroH}) have the same sign, while in the standard case they do not. We infer that the `wrong' sign is in the linear term looking at the corresponding finite horizon problem in \cite{AcquistapaceGozzi17}.}
of (\ref{eq:algriccintro}) (or \eqref{eq:algriccintroX}, or \eqref{eq:algriccintroH}) does not allow us to approach it using the standard method (described e.g. in \cite[pp. 390-394 and 479-486]{BDDM07}, see also
\cite[p.1018]{PriolaZabczyk03}), which consists in
taking the associated evolutionary Riccati equation, proving that it
has a solution $P(t)$ (using an {\it a priori} estimate, due to the fact that of both the linear and the quadratic terms have the same sign), and taking the limit of $P(t)$ as $t\to \infty$.

On the other hand the `wrong' sign comes from the nature of the
motivating problem: to look at minimum energy paths from
equilibrium to non-equilibrium states (see Section \ref{PHYSICS}),
which is the opposite direction of the standard one considered
in the above quoted papers.
This means that the value function depends on the final point, while
in the above quoted problems it depends on the initial one (see Remark
\ref{rm:timereversed} to see what happens to our auxiliary problem using a time inversion).
Therefore we are driven to use a different approach, that exploits
the structure of the problem; we partly borrow some ideas from \cite{PriolaZabczyk03}
and from the literature about model reduction\footnote{We
thank prof. R. Vinter for providing us these references.} (see e.g. \cite{Moore81} and \cite{Scherpen93}: indeed our results partly
generalize Theorem 2.2 of \cite{Scherpen93}, see Remark \ref{rm:findimScherpen}).



\subsection{The method and the main results}
\label{SSE:METHOD}

We now briefly explain our approach.
First of all we consider the associated finite horizon problem (which has been studied in the companion paper \cite{AcquistapaceGozzi17} whose results, for the part needed here, are recalled in Appendix \ref{sub:finitehorizon}), where the
state equation is
\begin{equation}\label{eq:stateintrofinhor}
\left\{ \begin{array}{ll} y'(r)=Ay(r)+Bu(r), \quad r\in \,]-t,0], \\[2mm]
y(-t) = 0, \end{array}\right.
\end{equation}
and the energy to be minimized is
\begin{equation}\label{eq:funzintrofinhor}
J_{[-t,0]}(u) = \frac12\int_{-t}^0 \|u(r)\|^2_U\, \ud r.
\end{equation}
The value function is
\begin{equation}\label{eq:ValueFunctionFinHor}
V(t,x)=\inf_{u\in {\cal U}_{[-t,0]}(0,x)} J_{[-t,0]}(u),
\end{equation}
where
\begin{equation}\label{eq:defcalUnew}
{\cal U}_{[-t,0]}(0,x)=\{ u \in L^2(-t,0;U): \; y(0)=x\}.
\end{equation}
We now recall the well known expression of the controllability operator
\begin{equation}\label{eq:contropintro}
Q_t x= \int_0^t e^{rA}BB^*e^{rA^*}x\, \ud r, \quad x\in X, \quad t\in [0,+\infty ].
\end{equation}
It is well known (see e.g \cite[Part IV, Theorem 2.3]{Zabczyk92}) that, for $t\ge 0$, the reachable set of the control systems \eqref{eq:stateintrofinhor} (finite horizon case)
and \eqref{eq:state-inf-new} (infinite horizon case),
is $\mathcal {\cal R}(Q^{1/2}_t)$, i.e. the range of $Q_t^{1/2}$ ($t \in [0,+\infty]$).
This is clearly the set where the value functions $V$ of \eqref{eq:ValueFunctionFinHor}
(finite horizon case) and $V_\infty$ of \eqref{eq:VFinfhornew} (infinite horizon case) are well defined. Moreover, as pointed out, e.g., in \cite[Proposition C.2-(i)]{AcquistapaceGozzi17}, for $0\le t_1\le t_2$ we have
$\mathcal {\cal R}(Q^{1/2}_{t_1})\subseteq \mathcal {\cal R}(Q^{1/2}_{t_2})$.
It will be often useful to assume, beyond Hypothesis \ref{hp:main}, also the following
null controllability assumption.
\begin{Hypothesis}\label{NC}
There exists $T_0\ge 0$ such that
\begin{equation}\label{eq:null-control}
\calr(e^{T_0A})\subseteq \calr(Q_{T_0}^{1/2}).
\end{equation}
\end{Hypothesis}
Under such assumption we get, for $t\ge T_0$,
$$
\mathcal {\cal R}(Q^{1/2}_{t_1})= \mathcal {\cal R}(Q^{1/2}_{t_2}), \qquad  T_0\le t_1\le t_2\le +\infty.
$$
Consequently
\begin{equation}\label{eq:kerNC}
\ker Q_{t_1}=\ker Q^{1/2}_{t_1}= \ker Q^{1/2}_{t_2}=\ker Q_{t_1}\,,
\qquad T_0\le t_1\le t_2\le +\infty.
\end{equation}
We can now introduce the already announced space $H$.
We define
\begin{equation}\label{eq:defH}
H=\mathcal {\cal R}(Q_\infty^{1/2}).
\end{equation}
Of course it holds
$$H \subseteq \overline{\mathcal {\cal R}(Q_\infty^{1/2})}= [\ker Q_\infty^{1/2}]^\perp =[\ker Q_\infty]^\perp.$$
The inclusion is in general proper.
Define in $H$ the inner product
\begin{equation}\label{eq:innerproductH}
\langle x, y\rangle_H = \langle Q_\infty^{-1/2}x,Q_\infty^{-1/2}y \rangle_X\,, \qquad x,y\in H.
\end{equation}
Some useful results on the space $H$ which form the ground for our main results and are partly proved in \cite{AcquistapaceGozzi17}, are recalled (and proved, when needed) in Appendix \ref{SSE:SPACEH}.

Using $H$ as the ground space we know (see \cite[Proposition 4.8-(ii)]{AcquistapaceGozzi17}) that $V(t,x)= \frac12\langle P(t)x,x \rangle_H$\,,
where $P(t)$ is a suitable extension of $Q_\infty Q_t^{-1}$ (here $Q_t^{-1}$ is the pseudoinverse of $Q_t$, see \cite[Appendix A]{AcquistapaceGozzi17} or \cite[Part IV, end of Section 2.1]{Zabczyk92}).
Moreover, using this explicit expression it is proved in  \cite[Theorem 4.12]{AcquistapaceGozzi17} that $P(t)$ solves the following Riccati equation in $H$:
\begin{equation} \label{eq:RiccatiHintro}
\frac{d}{dt}\langle P(t)x,y \rangle_H \hspace{-1mm}=\hspace{-1mm} -\langle Ax, P(t)y\rangle_H\hspace{-1mm} -\hspace{-1mm}\langle P(t)x, Ay\rangle_H\hspace{-1mm} -\hspace{-1mm} \langle {B}^*Q_\infty^{-1}P(t)x,{B}^*Q_\infty^{-1}P(t)y \rangle_U,\quad t>0,
\end{equation}
whose natural condition at $t=0$ is, heuristically,
$$
\lim_{t \to 0^+}P(t)=+\infty.
$$
It is not difficult to prove (see
Proposition \ref{pr:VlimitVt}) that
\begin{equation}\label{eq:VinftylimVt}
V_\infty (x)= \lim_{t \to + \infty} V(t,x).
\end{equation}
and that $V_\infty (x) = \frac12\langle  x,x \rangle_H$. This allows to prove that $P=I_H$ (the identity on $H$) solves the ARE (\ref{eq:algriccintro}) in $H$ (Theorem \ref{th:maximalARE}-(ii)).

However, due to the infinite initial condition at $0$ of $P(t)$ (similarly to what happens in \cite{PriolaZabczyk03}), the above limit
does not help to prove any comparison theorem for (\ref{eq:algriccintro}).
Here comes the main difficulty since, even in very simple cases, it is not known in the literature whether the ARE characterizes $V_\infty $ or not (see e.g. \cite{BDGJL4}). To get a comparison result we proceed
as follows.
\begin{itemize}
\item
We first introduce a suitable auxiliary problem (beginning of Section \ref{sub:ausprob}).
\item
Next, we prove a comparison result for the auxiliary problem (Subsection \ref{sub:MaxPN}, Lemma \ref{lem:massimalitaN}).
\item
Finally we use the relation among the auxiliary problem and the original problem
to prove our main maximality result (Theorem \ref{th:maximalARE}).
\end{itemize}
The idea of introducing an auxiliary problem is exploited in \cite{PriolaZabczyk03}, too.
However the method used there cannot work here, due to the different sign of the linear part of our equation.

%
%
%
%
%

\section{The auxiliary problem}\label{sub:ausprob}


In this section we introduce an auxiliary problem which can be considered a ``time reversed'' version of the auxiliary problem considered in \cite{PriolaZabczyk03} (see also Remark \ref{rm:timereversed} about this).
This problem will be a key tool to prove the main result, Theorem \ref{th:maximalARE}.
Indeed, as we will see, any solution of our Algebraic Riccati Equation (\ref{eq:algriccintro}) is also, under appropriate assumptions, a solution of this auxiliary problem with itself as initial datum; a comparison argument will then allow to get the main result.\\
Throughout this section Hypothesis \ref{hp:main} will be always assumed, while Hypothesis \ref{NC} will be used when necessary.\\
Let us consider, for $x \in X$, the following set of controls:
\begin{equation}\label{eq:defUbarx}
\overline{{\cal U}}_{[-t,0]}(x) = \{(z,u)\in H\times L^2(-t,0;U):
y(0)=x\},
\end{equation}
where $y(\cdot):=y(\cdot;-t,z,u)$ is the solution of the Cauchy problem (similar to \eqref{eq:stateintrofinhor} but with generic initial datum $z$)
\begin{equation} \label{eq:state-newN}
\left\{
\begin{array}{l}
y'(r)=Ay(r)+Bu(r), \quad r\in \,]-t,0], \\
y(-t) = z.
\end{array}
\right.
\end{equation}
Note that a control in $\overline{{\cal U}}_{[-t,0]}(x)$ is a couple: an initial point $z\in H$ and a control $u\in {\cal U}_{[-t,0]}(z,x)$, where
\begin{equation}\label{eq:defcalUnewbis}
{\cal U}_{[-t,0]}(z,x)=\{ u \in L^2(-t,0;U): \; y(0;-t,z,u)=x\}.
\end{equation}
(this is similar to the set (\ref{eq:defcalUnew}) but with a generic initial datum $z$). The following is true:

\begin{Proposition}\label{pr:Ubarnonempty}
Define the reachable set from the point $z$ as
\begin{equation}\label{eq:ptiragg-newprima}
{\mathbf R}_{[-t,0]}^z:= \left\{ x \in X:\ {\cal U}_{[-t,0]}(z,x) \neq \emptyset \right\}.
\end{equation}
and set
\begin{equation}\label{eq:defRbarra}
\bar{\mathbf R}_{[-t,0]} := \bigcup_{z\in H}
{\mathbf R}^z_{[-t,0]}.
\end{equation}
Then the set $\overline{{\cal U}}_{[-t,0]}(x)$ introduced in \eqref{eq:defUbarx}
is nonempty if and only if $x\in\bar{\mathbf R}_{[-t,0]}$.
Moreover we have
\begin{equation}\label{eq:RbarrainH}
\bar{\mathbf R}_{[-t,0]} \subseteq H,
\end{equation}
with equality for $t\ge T_0$, if Hypothesis \ref{NC} holds.
\end{Proposition}

\begin{proof} The first statement is an immediate consequence
of the definition of reachable set
in (\ref{eq:ptiragg-newprima}).
The second one follows from (\ref{eq:ptiragg-newbis}), Lemma \ref{exptAHH}-(i),
the fact that $\calr(Q_{t}^{1/2})\subseteq \calr(Q_{\infty}^{1/2})$
(with equality, for $t\ge T_0$, when Hypothesis \ref{NC} holds),
and the equality
\begin{equation}
\label{eq:ptiragg-Q12} {\cal R}\left(\call_{-t,0}\right)=
{\mathbf R}^0_{[-t,0]}=\calr(Q_{t}^{1/2}), \qquad  t \in [0,+\infty]
\end{equation}
(here $\call_{-t,0}$ is the operator defined in \eqref{eq:newLst}),
which is proved in \cite[Theorem 2.3]{Zabczyk92} for $t<+\infty$.
Such equality holds also when $t=+\infty$ with exactly the same proof.
\end{proof}


Given a bounded selfadjoint positive operator $N$ on $H$ we want to minimize, in the class $\overline{{\cal U}}_{[-t,0]}(x)$, the
following functional with an initial cost:
\begin{equation}\label{funzN}
J^N_{[-t,0]}(z,u) = \frac12 \langle Nz,z \rangle_H + \frac12 \int_{-t}^0 \|u(s)\|_U^2\, \ud s.
\end{equation}
The presence of the operator $N\in \call_+(H)$ forces us to fix the starting point $z$ at time $-t$ in $H$, rather than in $X$.
Define
\begin{equation}\label{eq:defVN-new}
V^N(t,x)=\inf_{(z,u) \in \overline{{\cal U}}_{[-t,0]}(x)} J^N_{[-t,0]}(z,u) = \inf_{z\in H} \left[ \inf_{u\in {\cal U}_{[-t,0]}(z,x)}
J^N_{[-t,0]}(z,u) \right], \ t>0, \ x\in X,
\end{equation}
with the agreement that the infimum over the emptyset is $+\infty$,
so that $V^N(t,x)$ is finite only when $x\in H$.
Now we provide a relation between $V^N$ and the value function $V$ defined in
\eqref{eq:ValueFunctionFinHor}.

\begin{Proposition}\label{pr:VNeV}
We have
\begin{equation}\label{eq:defVN-newbis}
V^N(t,x)=\inf_{z\in H} \left[ V(t,x-e^{tA}z) + \frac12 \langle Nz,z
\rangle_H  \right], \quad t>0, \ x\in X
\end{equation}
and, in particular,
\begin{equation} \label{eq:V0etVN-new}
V^N(t,x) \le V(t,x)
\qquad \forall x\in X, \quad \forall t>0.
\end{equation}
\end{Proposition}

\begin{proof}
We use \eqref{eq:valuefunction-new}, (\ref{eq:V1V1}) and \eqref{eq:defV0-new}
getting
$$
\inf_{u\in {\cal U}_{[-t,0]}(z,x)} J^N_{[-t,0]}(z,u)
= V_1(-t,0;z,x) + \frac12 \langle Nz,z \rangle_H = V(t,x-e^{tA}z) + \frac12 \langle Nz,z \rangle_H\,.
$$
This equality immediately implies \eqref{eq:defVN-newbis}.
Taking $z=0$ we get \eqref{eq:V0etVN-new}.
\end{proof}


The following remark is crucial to understand what is the ``natural'' Riccati equation associated to this auxiliary problem.

\begin{Remark}
\label{rm:timereversed}
{\em If $A$ generates not just a $C_0$-semigroup but a $C_0$-group,
the auxiliary problem can be shown, under appropriate assumptions, to be equivalent, reversing the time, to a standard optimization problem with final cost. Indeed, given $x \in H$, consider the problem of minimizing, over all
$v(\cdot) \in L^2(0,t;U)$, the functional
\begin{equation}\label{funzNrev}
\widehat J^N_{[0,t]}(x,v) = \frac12 \langle Nw(t),w(t) \rangle_H + \frac12 \int_0^{t} \|v(s)\|_U^2\, \ud s,
\end{equation}
where $w(\cdot):=w(\cdot;0,x,v)$ is the mild solution of the Cauchy problem
\begin{equation}\label{eqNrev}
w'(s)= -Aw(s)+Bv(s), \quad s\in \,]-t,0], \qquad w(0)=x.
\end{equation}
Assume now that, for every $x \in H$, the mild solution $w(\cdot;0,x,v)$
belongs to $H$ for every $t>0$.
Setting
$$\widehat V^N(t,x)= \inf_{v\in L^2(0,t;U)}\widehat J^N_{[0,t]}(x,v),$$
it can be seen that
$$\widehat V^N(t,x)=V^N(t,x).$$
To see this, fix $(t,x)\in [0,+\infty[ \,\times H$ and recall that, for every
$(z,u)\in \overline{\cal U}_{[-t,0]}(x)$, we have
$$e^{tA}z+\int_{-t}^0 e^{-sA}Bu(s)ds=x \quad \Longleftrightarrow \quad
z+\int_{-t}^0 e^{(-t-s)A}Bu(s)ds=e^{-tA}x;$$
hence, changing variable in the integral,
$$z=e^{t(-A)}x +\int_0^{t} e^{(t-s)(-A)}B(-u(-s))ds.$$
This means that $\bar {\mathbf{R}}_{[-t,0]}=H$ (see \eqref{eq:defRbarra}).
Moreover, to any $(z,u)\in \overline{\cal U}_{[-t,0]}(x)$ we can associate a function $v\in L^2(0,t;U)$ such that $w(t)=z$, namely, $v(s)=-u(-s)$; consequently
\beq\label{pb_equiv}
J^N_{[-t,0]}(z,u)= \widehat J^N_{[0,t]}(x,v).
\eeq
Conversely, given any $v\in L^2(0,t;U)$, set $z=w(t;0,x,v)$ and $u(s)=-v(-s)$: then, clearly, $(z,u)\in \overline{\cal U}_{[-t,0]}(x)$ and, again, \eqref{pb_equiv} holds. In conclusion, there is a one-to-one correspondence between the control set of the two problems and, in particular, $\widehat V^N(t,x)= V^N(t,x)$.}
\hfill\qedo
\end{Remark}
The equation for the ``time-reversed'' problem (\ref{funzNrev})-(\ref{eqNrev}) turns out to be the following:
\begin{equation}\label{eq:RiccatiNH}
\left\{\begin{array}{ll} \displaystyle \frac{d}{ds}\langle  P^N(s)x,y \rangle_H  = & -\langle Ax, P^N(s)y\rangle_H -\langle P^N(s)x, Ay\rangle_H
-\\[2mm]
& -\langle {B}^*Q_\infty^{-1} P^N(s)x,{B}^*Q_\infty^{-1} P^N(s)y \rangle_U\,,\qquad s\in\,]0,t], \\[2mm]
P^N(0)=N.
\end{array}\right.
\end{equation}
To give sense to \eqref{eq:RiccatiNH} we must take $x,y\in {\cal D}(A)\cap H$ with $Ax,Ay\in H$ and $P^N(t)x,P^N(t)y\in \calr(Q_\infty)$.
When ${B}^*Q_\infty^{-1}$ can be extended to a bounded operator $H \to U$
and $A$ generates a group, then it is known that the value function
$\widehat V^N$ is quadratic and $\widehat V^N(t,x)=\langle\widehat P^N(t)x,x\rangle_H$, where $\widehat P^N:[0,+\infty[ \to \Lc_+(H)$ is the unique solution of (\ref{eq:RiccatiNH}).
In our case this is not obvious, but it suggests anyway the right form of the Riccati equation for our auxiliary problem.
Note, finally, that the right hand side of (\ref{eq:RiccatiNH})
is exactly one of the forms of the ARE we aim to study (see \eqref{eq:algriccintroH}).

\begin{Remark}\label{Riccstrana2}
{\em As in the case $N=0$ treated in \cite{AcquistapaceGozzi17}, in the above Riccati equations the sign of the linear part is opposite to the usual one. In fact the control problem (\ref{eq:state-newN})-(\ref{funzN}) involves an ``initial cost'',
instead of a final cost like in the standard problems (see e.g. \cite{PriolaZabczyk03}). }
\hfill\qedo
\end{Remark}
Our aim now is to prove that for every stationary solution $Q$ of the Riccati equation \eqref{eq:RiccatiNH} (in a suitable class to be defined later) there exists an operator $N$, namely $Q$ itself, such that
$$\frac12 \<Qx,x\>_H \le V^N(t,x), \qquad \hbox{for sufficiently large $t$.}$$

\begin{Remark}
{\em
It is possible to prove much more about the auxiliary problem, namely:
\begin{itemize}
\item[(i)] that, for every $N\in \call_+(H)$ the value function $V^N$ is continuous and is a quadratic form in $H$;
\item[(ii)] that, when $N$ is coercive (i.e., for some $\nu>0$,
$\<Nx,x\>_H\ge \nu |x|^2_H$ for all $x \in H$), the linear operator $P^N$ associated to the value function solves the Riccati equation \eqref{eq:RiccatiNH};
\item[(iii)] that the comparison result mentioned above translates in the
inquality $P^N\ge Q^N$, in the preorder of positive operators, for every constant solution $Q^N$ of the Riccati equation \eqref{eq:RiccatiNH}
in a suitable class.
\end{itemize}
This is the subject of a paper in progress.}
\end{Remark}



\subsection{A key comparison result}
\label{sub:MaxPN}
Given any initial datum $N\in \Lc_+(H)$, we want to compare the
``stationary'' solutions of the Riccati equation (\ref{eq:RiccatiNH}) with the value function $V^N$ of the auxiliary problem.
This fact will be used, in the next section, as a key tool to prove our main results.
In order to do this we need first to give a precise meaning to the concept of stationary solution of (\ref{eq:RiccatiNH}).

Roughly speaking, a stationary solution $P\in \Lc_+(H)$ of the Riccati Equation \eqref{eq:RiccatiNH} should also be a solution of the following
Algebraic Riccati Equation (ARE), which comes from the right hand side of \eqref{eq:RiccatiNH}:
\beq
\label{eq:AREperLemmaMaximality0}
0=-\langle Ax,Py\rangle_H - \langle Px,Ay\rangle_H
-\langle B^*Q_\infty^{-1}Px,B^*Q_\infty^{-1}Py\rangle_U.
\eeq
This equation is meaningful for every $x,y\in {\cal D}(A)\cap H$ with $Px,Py\in {\cal R}(Q_\infty)$ and $Ax, Ay \in H$.
Since the last requirement appears too restrictive, we
rewrite \eqref{eq:AREperLemmaMaximality0} by taking the first two inner products in $X$, getting:
\beq
\label{eq:AREperLemmaMaximality}
0=-\langle Ax,Q_\infty^{-1}Py\rangle_X - \langle Q_\infty^{-1}Ox,Ay\rangle_X -\langle B^*Q_\infty^{-1}Px,B^*Q_\infty^{-1}Py\rangle_U.
\eeq
This makes sense in a larger set of vectors $x,y$, namely for every $x,y\in {\cal D}(A)\cap H$ with $Px,Py\in {\cal R}(Q_\infty)$.\footnote{Note that \eqref{eq:AREperLemmaMaximality0} is the same as \eqref{eq:algriccintroH} while \eqref{eq:AREperLemmaMaximality} is the same as \eqref{eq:algriccintro}.}
We can now provide the precise definition of solution of  \eqref{eq:AREperLemmaMaximality}.

\begin{Definition}\label{df:hplemmachiave}
Let $P\in\call_+(H)$ and define the operator $\Lambda_P$ as follows:
\beq\label{Lambda_P}\left\{\begin{array}{l} {\cal D}(\Lambda_P) = \{x\in H: \ Px\in {\cal R}(Q_\infty)\} \\[2mm]
\Lambda_P x = Q_\infty^{-1}Px \qquad \forall x\in {\cal D}(\Lambda_P).\end{array}\right.
\eeq
We say that $P$ is a solution of \eqref{eq:AREperLemmaMaximality}
(or, alternatively, a stationary solution of \eqref{eq:RiccatiNH})
if
${\cal D}(A)\cap {\cal D}(\Lambda_P)$ is dense in $[\ker Q_\infty]^\perp$ and
\beq\label{ARE_chiave}
0=-\langle Ax,\Lambda_P y\rangle_X - \langle \Lambda_P x , Ay\rangle_X - \langle B^*\Lambda_P x, B^*\Lambda_P y\rangle_U \qquad \forall x,y\in
{\cal D}(A)\cap{\cal D}(\Lambda_P).
\eeq
\end{Definition}

We now define a subclass $\calq$ of the class of all stationary solutions of \eqref{eq:RiccatiNH}. First of all we recall that, by Lemma \ref{exptAHH}-(i), $e^{tA}|_H$ is a strongly continuous semigroup in $H$. We then use the following notation.
\begin{Notation}\label{eq:A0}
We denote by $A_0:{\cal D}(A_0)\subseteq H \to H$ the generator
of $e^{tA}|_H$, and we write $e^{tA_0}$ in place of $e^{tA}|_H$.
\end{Notation}

\begin{Definition}\label{hp:hplemmachiave}
Let $P\in \Lc_+(H)$. We say that $P \in \calq$ if there exists $D\subseteq \cald(\Lambda_P)$ such that $D$ is dense in $\cald (A)\cap H$ with respect to the norm $\|\cdot\|_H+\|A\cdot\|_X$;
\end{Definition}

\begin{Lemma}\label{lm:hplemmachiaveI}
The set $\calr(Q_\infty)\cap \cald(A)$ is dense in $\cald(A)\cap H$, equipped with the norm $\|\cdot\|_H+\|A\cdot\|_X$. Hence, choosing $D=\calr(Q_\infty)\cap \cald(A)$, we have
$P=I_H\in \calq$.
\end{Lemma}

\begin{proof}
Let $x \in H\cap \cald(A)$ such that
$$
\< x,z\>_H+\< Ax,Az\>_X=0, \qquad \forall z \in \calr(Q_\infty)
\cap \cald(A).
$$
It is enough to prove that $x=0$. Observe that, writing $z=Q_\infty y$,
$$
\< x,Q_\infty y\>_H+\< Ax,AQ_\infty y\>_X=0,
\qquad \forall y \in \cald(AQ_\infty).
$$
Then
$$
\< Ax,AQ_\infty y\>_X=- \< x,Q_\infty y\>_H= -\< x,y\>_X
\qquad \forall y \in \cald(AQ_\infty).
$$
This means that $Ax \in \cald((AQ_\infty)^*)$ and $(AQ_\infty)^*Ax=-x$. Hence
$$
\<(AQ_\infty)^* Ax, Ax\>_X=- \< x,Ax\>_X= |(-A)^{1/2}x|_X^2 \ge 0.
$$
On the other hand we know, from \cite[Lemma 3.1-(ii)]{AcquistapaceGozzi17}, that, for every
$y\in\cald((AQ_\infty)^*)\subseteq \cald(AQ_\infty)$
$$
2\<(AQ_\infty)^*y,y\>_X  =-\|B^*y\|^2_U\,,
$$
so that
$$
2\<(AQ_\infty)^* Ax, Ax\>_X = -\|B^*Ax\|^2_U \le 0.
$$
This implies that $\|(-A)^{1/2}x\|_X^2 =0$; hence $Ax=0$ and, since $A$ is invertible, $x=0$.
\end{proof}
\begin{Lemma}\label{lem:massimalitaN}
Assume Hypothesis \ref{NC}. Let $P\in \Lc_+(H)$ be a solution of \eqref{eq:AREperLemmaMaximality} according to Definition \ref{df:hplemmachiave}. Assume also that $P \in \calq$ and that
$BB^*$ is coercive, which is equivalent to require that, for some $\mu>0$, $\|B^*x\|_U\ge \mu\|x\|_X$  for all $x \in X$.
Then, the following estimate holds:
$$\frac12\langle Px,x\rangle_H \le V^P(t-T_0,x) \qquad \forall x\in H,\ \
\forall t>T_0,$$
where $V^P$ is the value function defined in \eqref{eq:defVN-new} with $N=P$.
\end{Lemma}

\begin{proof} \rule{0mm}{1mm}\\
{\bf Step 1} We prove the estimate
\beq\label{stima_aux}
\langle Px,x\rangle_H\le
\langle Py(T_0-t),y(T_0-t)\rangle_H
+\int_{T_0-t}^0 \|u(s)\|_U^2\,\ud s, \qquad t>T_0,
\eeq
for every $(z,u)\in \overline{{\cal U}}_{[-t,0]}(x)$ with $x\in H$, where $y$ is the state corresponding to $(z,u)$, i.e.
\begin{equation}\label{eq:StatoDimLemmaChiave}
y(s) = e^{(s+t)A}z + \int_{-t}^s e^{(s-\sigma)A} \, Bu(\sigma)\,d\sigma, \quad s\in [-t,0].
\end{equation}
Such inequality would be easy to prove if we were able to compute
$\frac{d}{ds}\langle Py(s),y(s)\rangle_H$
and prove that
$$
\frac{d}{ds}\langle Py(s),y(s)\rangle_H\le \|u(s)\|^2_U, \qquad s \in [-t,0].
$$
Unfortunately we even do not know if such a derivative exists.
Hence we need to build a delicate approximation procedure as follows.

Fix $t>T_0$ and $x\in H$; consider any $(z,u)\in \overline{{\cal U}}_{[-t,0]}(x)$. It is not restrictive to assume in \eqref{eq:StatoDimLemmaChiave} that $u(\sigma)\in \overline{{\cal R}(B^*)}$ for every $\sigma \in [-t,0]$: indeed, writing, for every such $\sigma$,
$$u(\sigma)=u_1(\sigma)+u_2(\sigma), \quad u_1(\sigma) \in  \overline{{\cal R}(B^*)},\quad u_2(\sigma) \in  \overline{{\cal R}(B^*)}^\perp=\ker B,$$
it is clear that $e^{(s-\sigma)A}Bu_2(\sigma)=0$. Hence
$$y(s) = e^{(s+t)A}z + \int_{-t}^s e^{(s-\sigma)A} \, Bu_1(\sigma)\,d\sigma, \quad s\in [-t,0].
$$
Since, evidently, $J^P_{[-t,0]}(z,u)\ge J^P_{[-t,0]}(z,u_1)$, we can always choose $u_1$ in place of $u$.
Next, select a sequence $\{(z_n,u_n)\}\subseteq \big[{\cal D}(A_0)\big] \times C^1_0([-t,0];U)$\footnote{$C^1_0([-t,0];U)$ is the set of $C^1$ $U$-valued functions which take the value $0$ at the boundary.}, such that $u_n$ is ${\cal R}(B^*)$-valued and $(z_n,u_n)\to (z,u)$ in $H\times L^2(-t,0;U)$.
Thus we can set $u_n=B^*v_n$, where $v_n\in C^1_0([-t,0],X)$ and,
denoting by $y_n$ the corresponding state, we have $y_n\in C^1([-t,0];H) \cap C([-t,0];\cald(A))$ (see e.g. \cite[Chapter 4, Corollary 2.5]{Pazy83}) and
$$
y_n(s) = e^{(s+t)A}z_n + \int_{-t}^s e^{(s-\sigma)A} \, BB^*v_n(\sigma)\,d\sigma , \qquad s\in [-t,0].
$$

%
%
Thanks to the properties of the set $D$ of Hypothesis \ref{hp:hplemmachiave}, we can now choose, for every $n\in \N$, another approximating sequence
$\{y_{nk}\}_{h \in N} \subset C^1([-t,0],H)\cap C^0_0([-t,0],\cald(A))$, such that $y_{nk}(s) \in D$ for every $s \in [-t,0]$ and satisfying, as $k \to +\infty$,
\begin{equation}\label{eq:convnkh1}
y_{nk} \to y_{n} \hbox{ in } C^1([-t,0];H), \qquad
Ay_{nk} \to Ay_{n} \hbox{ in } C([-t,0];X)
\end{equation}
(see e.g. \cite[Chapter 4, Theorem 2.7]{Pazy83}).
Set now $w_{nk}=y'_{nk}-Ay_{nk}$. By \eqref{eq:convnkh1} we get, for every $n\in \N$,
\begin{equation}\label{eq:convnkh1bis}
w_{nk} \to y_{n}-Ay_{n} = BB^*v_n \hbox{ in } C^0([-t,0];X) \qquad
\hbox{as $k \to +\infty$.}
\end{equation}
%
We now can differentiate the quantity
$\langle Py_{nk}(s), y_{nk}(s)\rangle_H$ for $s\in [-t,0]$. Indeed, taking into account the above definition of $w_{nk}$, we obtain, for $s\in [-t,0]$ and $n,k\in\N$:
\bey
\lefteqn{\frac{d}{ds} \langle Py_{nk}(s), y_{nk}(s)\rangle_H = \langle y_{nk}'(s),Py_{nk}(s)\rangle_H + \langle Py_{nk}(s),y_{nk}'(s)\rangle_H =}\\
& & =\langle y_{nk}'(s), \Lambda_P y_{nk}(s)\rangle_X + \langle \Lambda_P y_{nk}(s),y_{nk}'(s)\rangle_X = \\[1mm]
& & = \langle Ay_{nk}(s)+w_{nk}(s),\Lambda_P y_{nk}(s)\rangle_X + \langle \Lambda_P y_{nk}(s),Ay_{nk}(s)+w_{nk}(s)\rangle_X .
\eey
Since $P$ solves the ARE \eqref{ARE_chiave} we get, for every $s\in [-t,0]$,
\bey
\lefteqn{\frac{d}{ds} \langle Py_{nk}(s), y_{nk}(s)\rangle_H =}\\[1mm]
& & = -\|B^*\Lambda_P y_{nk}(s)\|_U^2 + \langle w_{nk}(s),\Lambda_P y_{nk}(s)\rangle _X + \langle \Lambda_P y_{nk}(s),w_{nk}(s)\rangle _X =\\[2mm]
& & = -\|B^*\Lambda_P y_{nk}(s)\|_U^2 + \langle B^*v_n(s),B^*\Lambda_P y_{nk}(s)\rangle _U + \langle B^*\Lambda_P y_{nk}(s),B^*v_n(s)\rangle _U =
\\[2mm]
& & \quad + \langle w_{nk}(s)-BB^*v_n(s),\Lambda_P y_{nk}(s)\rangle _X +\langle \Lambda_P y_{nk}(s),w_{nk}(s)-BB^*v_n(s)\rangle _X =
\\[2mm]
& & = -\|B^*\Lambda_P y_{nk}(s)-B^*v_n(s)\|_U^2 + \|B^*v_n(s)\|_U^2 +
\\[2mm]
& & \quad + \langle w_{nk}(s)-BB^*v_n(s),\Lambda_P y_{nk}(s)\rangle _X + \langle \Lambda_P y_{nk}(s),w_{nk}(s)-BB^*v_n(s)\rangle _X .
\eey
Hence, recalling that $u_n=B^*v_n$, we may write for every $\varepsilon>0$,
\begin{equation}\label{eq:stimanew}
\begin{array}{lcl} \displaystyle \frac{d}{ds} \langle Py_{nk}(s), y_{nk}(s)\rangle_H & \le & -\|B^*\Lambda_P y_{nk}(s)-B^*v_n(s)\|_U^2 +\|u_n(s)\|_U^2 +\\[3mm]
& & + 2\|w_{nk}(s)-Bu_n(s)\|_X \|\Lambda_QPy_{nk}(s)\|_X \le\\[2mm]
& \le & -\|B^*\Lambda_P y_{nk}(s)-B^*v_n(s)\|_U^2 +\|u_n(s)\|_U^2 +\\[2mm]
& &  \displaystyle + \frac1{\varepsilon} \|w_{nk}(s)-Bu_n(s)\|_X^2
+\varepsilon \|\Lambda_P y_{nk}(s)\|_X^2.
\end{array}
\end{equation}
Now observe that
$$
\varepsilon
\|\Lambda_P y_{nk}(s)\|_X^2\le
\frac \eps\mu
\|B^*\Lambda_P y_{nk}(s)\|_U^2\le
2\frac \eps\mu
\|B^*\Lambda_P y_{nk}(s)-B^*v_n(s)\|_U^2
+ 2\frac \eps\mu
\|B^*v_n(s)\|_U^2.
$$
Inserting this inequality into \eqref{eq:stimanew} we get
\begin{equation}\label{eq:stimanewbis}
\begin{array}{lcl} \displaystyle \frac{d}{ds} \langle Py_{nk}(s), y_{nk}(s)\rangle_H & \le & \displaystyle -\left(1-2\frac\eps\mu \right) \|B^*\Lambda_P y_{nk}(s)-B^*v_n(s)\|_U^2+ \\[4mm]
& & \displaystyle +\left(1+2\frac\eps\mu \right)\|u_n(s)\|_U^2 + \frac1{\varepsilon} \|w_{nk}(s)-Bu_n(s)\|_X^2.
\end{array}
\end{equation}
Hence, for all positive $\eps$ such that $2\frac\eps\mu\le \frac12$
we get
\begin{equation}
\label{eq:stimanewter}
\frac{d}{ds} \langle Py_{nk}(s), y_{nk}(s)\rangle_H \le
\left(1+2\frac\eps\mu \right)\|u_n(s)\|_U^2
+ \frac1{\varepsilon} \|w_{nk}(s)-Bu_n(s)\|_X^2\,.
\end{equation}
Now we have for every $s\in [-t,0]$, as $k\to \infty$,
$$\|y_{nk}(s)-y_n(s)\|_H\to 0, \quad \|y_{nk}'(s)-y_n'(s)\|_H\to 0, \quad \|w_{nk}(s) -Bu_n(s)\|_X \to 0;$$
thus we get, for every $n\in \N^+$, $s\in [-t,0]$ and $0<\varepsilon \le \mu/4$,
$$
\frac{d}{ds} \langle Py_n(s), y_n(s)\rangle_H
\le \left(1+2\frac\eps\mu \right)\|u_n(s)\|_U^2\,.
$$
Finally, letting $\varepsilon\to 0$,
$$\frac{d}{ds} \langle Py_n(s), y_n(s)\rangle_H \le \|u_n(s)\|_U^2 \quad \forall n\in \N^+, \quad \forall s\in [-t,0].$$
We now integrate in the smaller interval $[T_0-t,0]$:
$$\langle P y_n(0),y_n(0)\rangle_H \le \langle Py_n(T_0-t),y_n(T_0-t)\rangle_H + \int_{T_0-t}^0 \|u_n(s)\|_U^2\,\ud s.$$
Letting $n\to \infty$, since $y_n(s) \to y(s)$ for every $s\in [-t,0]$, $y(0)=x$, and $u_n\to u$ in $L^2(-t,0;U)$, we deduce for every $(z,u)\in \overline{{\cal U}}_{[-t,0]}(x)$
$$\langle Px,x\rangle_H \le \langle Py(T_0-t),y(T_0-t)\rangle_H + \int_{T_0-t}^0 \|u(s)\|_U^2\,\ud s, \qquad t>T_0;$$
this is equation \eqref{stima_aux}.\\[2mm]
{\bf Step 2} We complete the proof of the Lemma. Consider a sequence $(\hat z_n,\hat u_n)\in \overline{{\cal U}}_{[T_0-t,0]}(x)$, such that, as $n\to \infty$,
\beq\label{optimalQ}
J^P_{[T_0-t,0]}(\hat z_n,\hat u_n) \to \inf_{(z,u)\in \overline{{\cal U}}_{[T_0-t,0]}(x)} J^P_{[T_0-t,0]}(z,u) = V^P(t-T_0,x).
\eeq
Thus $\hat z_n\in H$, $\hat u_n\in L^2(T_0-t,0;U)$ and the corresponding state is
$$\hat y_n(s)= e^{(s+t-T_0)A}\hat z_n + \int_{T_0-t}^s e^{(s-\sigma)A} B\hat u_n(\sigma)\,\ud \sigma, \quad s\in [T_0-t,0];$$
in particular $\hat y_n(0)=x$. Now choose $\hat v_n \in L^2(-t,T_0-t;U)$ such that
\beq\label{incollo}
\int_{-t}^{T_0-t} e^{(T_0-t-\sigma)A}B\hat v_n(\sigma)\,\ud\sigma =\hat z_n;
\eeq
this is possible since, due to Hypothesis \ref{NC}, the range of the operator
(defined in \eqref{eq:newLst})
$$v\mapsto {\cal L}_{-t,T_0-t}(v) = {\cal L}_{-T_0,0}(v(\cdot +t-T_0))$$
is all of $H$ (see \cite[Theorem 2.3]{Zabczyk92}). Then, setting
$$\overline{u}_n = \left\{ \begin{array}{ll} \hat v_n & \textrm{in } [-t,T_0-t] \\[2mm]
\hat u_n & \textrm{in } [T_0-t,0], \end{array}\right.$$
the state corresponding to $(0,\overline{u}_n)$ in $[-t,0]$ is
$$\overline{y}_n(s) = \int_{-t}^s e^{(s-\sigma)A} B\overline{u}_n(\sigma)\,\ud\sigma.$$
By \eqref{incollo} we have
$$\overline{y}_n(T_0-t)=\int_{-t}^{T_0-t} e^{(T_0-t-\sigma)A} B\overline{u}_n(\sigma)\,\ud\sigma = \hat z_n;$$
hence, by uniqueness,
$$\overline{y}_n(s) = e^{(s+t-T_0)A}\hat z_n + \int_{T_0-t}^s e^{(s-\sigma)A} B\hat u_n(\sigma)\,\ud\sigma = \hat y_n(s) \qquad \forall s\in [T_0-t,0],$$
so that $\overline{y}_n(0)=\hat y_n(0)=x$. This shows that $(0,\overline{u}_n)\in \overline{\calu}_{[-t,0]}(x)$, and consequently, by \eqref{stima_aux},
$$\langle Px,x\rangle_H \le \langle P\hat z_n,\hat z_n\rangle_H + \int_{T_0-t}^0 \|\hat u_n(s)\|^2_U\,\ud s=2J^P_{[T_0-t,0]}(\hat z_n,\hat u_n).$$
Finally, by \eqref{optimalQ}, as $n\to \infty$ we get
$$\frac12\langle Px,x\rangle_H \le V^P(t-T_0,x) \qquad \forall t> T_0, \quad \forall x\in H.$$
\end{proof}

\section{Minimum energy with (negative) infinite horizon}
\label{sub:infhor}

We now give a precise formulation of our infinite horizon problem
(see Subsection \ref{SSE:introinfhor}
and also \cite[Remark 2.8]{AcquistapaceGozzi17}).
We assume that Hypothesis \ref{hp:main} holds throughout this section without repeating it.
For any given control $u\in L^2(-\infty,s;U)$ we take the state equation
\begin{equation}
\label{eq:state-inf-newbis} \left\{
\begin{array}{l}
y'(r)=Ay(r)+Bu(r), \quad r\in \,]-\infty,s], \\
y(-\infty) = 0.
\end{array}
\right.
\end{equation}
By Lemma \ref{lm:existencesol-new} we know that the unique solution of (\ref{eq:state-inf-newbis}) belongs to $C(\,]-\infty,s];X)$, is given by
$$
y(r):=y(r;-\infty,0,u)=\int_{-\infty }^{r}e^{(r-\tau)A}B u(\tau)\, \ud \tau,
\quad -\infty < r \le s,
$$
and satisfies, for every $-\infty < r_1 \leq r_2 \leq s$,
$$y(r_2) =  e^{(r_2 - r_1)A} y(r_1)+\int_{r_1}^{r_2} e^{(r_2-\tau)A} Bu(\tau)\, \ud \tau, \quad \textrm{and} \quad \lim_{r\to -\infty}  y(r)= 0 \quad \hbox{ in $X$}.$$
As for the finite horizon case, we define:
\begin{equation}\label{eq:defUcal}
{\cal U}_{[-\infty,s]}(0,x) \nd \left\{ u\in L^2(-\infty,s;U) \; :
\; y(s;-\infty,0,u)=x \right \},
\end{equation}
\[
J_{[-\infty,s]}(u) =\frac12 \int_{-\infty}^s \| u(r)\|^2_{U}\;\ud r,
\]
\[
V_1(-\infty,s;0,x)\nd\inf_{u\in {\cal U}_{[-\infty,s]}(0,x)}
J_{[-\infty,s]}(u),
\]
with the agreement that the infimum over the empty set is $+\infty$.
From (\ref{eq:defUcal}) it is easy to see that
\begin{equation}
\label{eq:spacontr-new-infhor} u(\cdot)\in {\cal
U}_{[-\infty,s]}(0,x) \ \iff \ u(\cdot-s)\in  {\cal
U}_{[-\infty,0]}(0,x);
\end{equation}
this implies that
$$
V_1(-\infty,s;0,x)=V_1(-\infty,0;0,x).$$
From now on we set, as in \eqref{eq:defV0-new}
\begin{equation}\label{eq:defV0-new-infhor}
V_\infty(x) = V_1(-\infty,0;0,x) = \inf_{u\in {\cal U}_{[-\infty,0]}(0,x)} J_{[-\infty,0]}(u), \quad x\in X.
\end{equation}
We collect now some results about the above problem and the function $V_\infty$.

\subsection{Optimal strategies}
\label{SSE:OPTSTRAT}
We start proving the existence of optimal strategies.

\begin{Proposition}\label{pr:exist}
The set ${\cal U}_{[-\infty,0]}(0,x)$ is nonempty if and only if $x \in H$.
Moreover, for every $x \in H$ there exists a unique
$\hat{u}_x\in {\cal U}_{[-\infty,0]}(0,x)$ such that
$$V_\infty (x)= J_{[-\infty,0]}(\hat u_x).$$
\end{Proposition}
\begin{proof}
The first statement follows from \eqref{eq:ptiragg-Q12} as in Proposition \ref{pr:Ubarnonempty}.
Now take $x \in H$ and observe that any minimizing sequence $\{u_n\}_{n\in {\mathbb N}}$ must be bounded
in $L^2(-\infty,0;U)$; so, passing to a subsequence, we have $u_n
\rightharpoonup \hat{u}_x$ in $L^2(-\infty,0;U)$. As the functional
$J_{[-\infty,0]}$ is weakly lower semicontinuous, we get
$$V_\infty(x) \le J_{[-\infty,0]}(\hat{u}_x) \le \liminf_{n\to \infty} J_{[-\infty,0]}(u_n) =V_\infty(x),$$
i.e. $\hat{u}_x$ is optimal. Uniqueness is an easy consequence of the strict convexity of the functional $J_{[-\infty,0]}$.
\end{proof}
Moreover we have the following result about the optimal couples when $x \in {\cal R}(Q_\infty)$ (see \cite[Proposition C.3 and Remark C.4]{AcquistapaceGozzi17}).

\begin{Proposition}\label{pr:optcoupleinfty}
Let $x \in {\cal R}(Q_\infty)$. Let $(\hat y_{x},\hat u_{x})$ be the optimal couple for our problem with target $x$. Then we have
\begin{equation}\label{eq:optcontrt}
\hat u_{x}(r) = B^* e^{-rA^*} { Q}^{-1}_{\infty} x, \quad r\in\,]-\infty,0].
\end{equation}
Moreover the corresponding optimal state $\hat y_x$ satisfies
\begin{equation}\label{eq:optstatet}
\hat y_{x}(r)=Q_{\infty}e^{-rA^*}Q_\infty^{-1}x, \quad r \in\,]-\infty,0];
\end{equation}
hence the optimal couple satisfies the feedback formula
\begin{equation}\label{eq:optfeedbackt}
\hat u_{x}(r)=B^* Q_{\infty}^{-1}\hat y_{x}(r), \quad r \in \,]-\infty,0],
\end{equation}
and, formally, $\hat y_{x}$ is a solution of the backward closed loop equation (BCLE)
\begin{equation}\label{eq:CLEt}
y'(r)=(A+BB^* Q_{\infty}^{-1}) y(r), \quad r \in\,]-\infty,0[\,, \quad y(0)=x,
\end{equation}
which, since $Q_\infty$ solves the Lyapunov equation
(see \cite[Proposition 3.3]{AcquistapaceGozzi17}
rewrites as
\begin{equation}\label{eq:CLEtinftybis}
y'(r)=-Q_{\infty}A^*Q_{\infty}^{-1} y(r), \quad r \in\,]-\infty,0[\,.
\end{equation}
If $A^*$ commutes with $Q_\infty$ (e.g. when $A$ is selfadjoint and invertible,
and $A$ and $BB^*$ commute), then \eqref{eq:CLEtinftybis} becomes
\begin{equation}\label{eq:CLEinftyA*}
y'(r)=-A^* y(r) , \quad r \in\,]-\infty,0[\,.
\end{equation}
This means that, in such case, the optimal trajectory arriving at $x$ is given by
$$y(r)=e^{-rA^*}x, \quad r \in\,]-\infty,0].$$
\end{Proposition}

\subsection{Connection with the finite horizon case}
\label{SSE:FININF}

We now prove the connection between $V_\infty$ and the value function $V$ of the
corresponding finite horizon problem which is studied in \cite{AcquistapaceGozzi17}
(see also Appendix A).
\begin{Proposition}\label{pr:VlimitVt}
Under Hypothesis \ref{NC}, for every $x \in H$ we have
\[
V_{\infty}(x) =\lim_{t\to +\infty}  V(t,x)=\inf_{t>0} V(t,x).
\]
Moreover $V_{\infty}(x)=\frac12 \|x\|_H^2$.
\end{Proposition}
\begin{proof}
First of all, by \cite[Proposition 4.8-(i)]{AcquistapaceGozzi17}, the function
$V(\cdot,x)$ is decreasing for every $x \in H$; hence, for every such $x$
$$\exists \,\lim_{t\to +\infty}  V(t,x) =\inf_{t>0} V(t,x).$$

We now prove that \(V_{\infty}(x) \le \inf_{t>0} V(t,x)\).
With an abuse of notation we can write
\[
{\cal U}_{[-t,0]}(0,x) \subseteq {\cal U}_{[-\infty,0]}(0,x) \qquad \forall t > 0:
\]
indeed, given a control bringing $0$ to $x$ in the interval $[-t,0]$,
we can extend it to a control bringing $0$ to $x$ in the interval
$[-\infty,0]$ just taking the null control on $\,]-\infty,-t]$.
So, if the set ${\cal U}_{[-t,0]}(0,x)$ is not empty, {\em a fortiori} the set ${\cal U}_{[-\infty,0]}(0,x)$ will be not empty.
This fact, together with the monotonicity of $V(\cdot,x)$ implies that
$V_{\infty}(x) \le \inf_{t>0} V(t,x)$.\\
We prove now that $V_{\infty}(x) = \inf_{t>0} V(t,x)$. Assume by contradiction that $V_{\infty}(x) < \inf_{t>0} V(t,x)$, and let $\eps >0$ be such that $V_{\infty}(x)+2\eps < \inf_{t>0} V(t,x)$. Take $u_\eps \in {\cal U}_{[-\infty,0]}(0,x)$ such that $J_{[-\infty,0]}(u_\eps)< V_\infty (x) + \eps$. By \eqref{eq:defsolution-new} we get
$$
x=\int_{-\infty}^0 e^{-\tau A}Bu_\eps(\tau)\, \ud\tau =e^{tA}y(-t) + \int_{-t}^{0} e^{-\tau A}Bu_\eps(\tau)\,\ud\tau \qquad \forall t>0;
$$
hence we have
$u_\eps|_{[-t,0]} \in {\cal U}_{[-t,0]}(y(-t),x)$,
which in turn implies that
\begin{equation}\label{eq:Vueps}
V(t,x-e^{tA}y(-t)) \le \frac12\int_{-t}^0 \|u_\eps(s)\|_U^2\,\ud s.
\end{equation}
Now we observe that for every $\delta\in\,]0,1[\,$ we may choose $t_\delta>T_0+1$ such that
$\|e^{tA}y(-t)\|_H \le \delta$ for every $t> t_\delta$: indeed, by Hypothesis \ref{NC} and Lemma \ref{HH}-(v) we have
\begin{eqnarray*}
\|e^{tA}y(-t)\|_H & = & \|Q_\infty^{-1/2}e^{tA}y(-t)\|_X
\le \|Q_\infty^{-1/2}e^A\|_{\Lc (X)} \|e^{(t-1)A}y(-t)\|_X \le \\[1mm]
& \le & \|Q_\infty^{-1/2}e^A\|_{\Lc (X)} M e^{-\omega(t-1)} \|y(-t)\|_X\,.
\end{eqnarray*}
Since $y(-t)$ is uniformly bounded in $X$ for $t>0$, we have the claim.\\
Going ahead with the proof, we recall that, by
\cite[Proposition 4.8-(iii)-(b)]{AcquistapaceGozzi17},
we have uniform continuity of $V$ on $[T_0,+\infty]\times B_H(0,R)$
for every $R>0$, where $B_H(0,R)$ is the ball of center $0$ and radius $R$ in $H$. So, setting $R=\|x\|_H + 1$, and denoting by $\rho_R$ the continuity modulus of $V$ on $[T_0,+\infty]\times B_H(0,R)$, we have for $t> t_\delta$
$$
V\left(t,x-e^{tA}y(-t)\right)> V(t,x) - \rho_R (\delta).
$$
The above, together with \eqref{eq:Vueps}, implies that
$$
V(t,x) - \rho_R (\delta)\le V_\infty (x) + \eps \qquad \forall t> t_\delta\,.
$$
Now it is enough to choose $\delta$ such that $ \rho_R (\delta)< \eps$ to get a contradiction.

Finally the last statement follows from
\cite[Proposition 4.8-(iii)-(d)]{AcquistapaceGozzi17}.
\end{proof}

\subsection{Algebraic Riccati Equation}
\label{SSE:ARE}

We deal with the Algebraic Riccati Equation (ARE from now on) associated to our infinite horizon problem. As well known, when the value function is a quadratic form in the state space $X$, the ARE is an equation whose unknown is an operator $R$. A typical goal in studying such ARE is to prove that the operator representing the quadratic form in $X$ given by the value function is a solution (possibly unique) of the associated ARE.
Formally our ARE is given as follows:
\begin{equation}\label{eq:algriccX}
0=-\langle Ax, Ry\rangle_X - \langle Rx, Ay\rangle_X
- \langle {B}^*Rx,{B}^*Ry \rangle_U, \qquad x,y\in \cald(A).
\end{equation}
In our case (see Proposition \ref{pr:exist}) the value function $V_\infty$ is finite only in $H$ so that the operator $R$ above must be unbounded
and the above equation makes sense only for $x,y\in \cald(A)\cap \cald(R)$.
Moreover, by Proposition \ref{pr:VlimitVt}, $V_\infty$ is a quadratic form on the space $H$, represented by the identity operator $I_H \in \call(H)$, i.e. $V_\infty(x)=\frac12 \|x\|^2_H$.
Consequently, transforming such norm in $X$, it must be $V_\infty(x)=\frac12\|Q_\infty^{-1/2}\|_X^2$, and, when
$x \in R(Q_\infty)$,
$V_\infty(x)=\frac12\<Q_\infty^{-1}x,x\>_X$.
Hence it is natural to deduce that the operator representing $V_\infty$ in the space $X$ is $Q_\infty^{-1}$.

Due to the unboundedness of the candidate solution $R$ of the ARE
\eqref{eq:algriccX}, it seems better to study the corresponding ARE in the space $H$, with unknown $P\in \call(H)$ whose form (taking $R=Q_\infty^{-1}P$) must be
(compare with \eqref{eq:algriccintro}):
\begin{equation}\label{eq:algricc}
0=-\langle Ax, Q^{-1}_\infty Py\rangle_X
- \langle Q^{-1}_\infty Px, Ay\rangle_X
-\langle {B}^*Q_\infty^{-1}Px,{B}^*Q_\infty^{-1}Py \rangle_U.
\end{equation}
Note that such expression makes sense only when
$Px,Py \in {\cal R}(Q_\infty)$ and
$x,y\in {\cal D}(A)\cap H$.

%

By Proposition \ref{pr:VlimitVt}, we expect that the positive selfadjoint operator $P=I_H$ associated with the value function $V_\infty$ is a solution of the above ARE \eqref{eq:algricc}. Similarly we expect that $R=Q_\infty^{-1}$ is a solution of the above ARE \eqref{eq:algriccX}.
As they cannot be unique (the zero operator is always a solution of both), we somehow expect such solutions to be maximal in some suitable sense.

\begin{Remark}\label{rm:findimScherpen} {\em
In the finite-dimensional case, when the operator $Q_{\infty}$ is invertible, it is proved that the operator $R=Q_\infty^{-1}$
solves (\ref{eq:algriccX}), using the fact that its inverse
$W=Q_{\infty}$ is the unique solution of the Lyapunov equation
\begin{equation}\label{Lyap}
AW+WA^*=-BB^*
\end{equation}
among all definite positive bounded operators $X\to X$. This is
reported by Scherpen \cite[Theorem 2.2]{Scherpen93}, who quotes
Moore \cite{Moore81} for the proof (see also \cite[Chapters 5 and 7]{LancasterARE95} for related results).
In fact, as we will see, this procedure works
in our infinite dimensional case, too, but with more difficulties.
}
\hfill\qedo
\end{Remark}


\begin{Definition}\label{df:solutionAREHX}
\begin{itemize}
\item[(i)] An operator $P\in \Lc_+(H)$ is a solution of the ARE (\ref{eq:algricc}) if the set $\cald(A)\cap \cald(\Lambda_P)$ (see \eqref{Lambda_P}) is dense in $H$ and the equation (\ref{eq:algricc}) is satisfied for all $x,y\in \cald(A)\cap \cald(\Lambda_P)$.
\item[(ii)] A positive, selfadjoint, possibly unbounded operator $R:{\cal D}(R)\subset X \to X$ is a solution of the ARE (\ref{eq:algriccX}) if the set $\cald(A)\cap \cald(R)$ is dense in $[\ker Q_\infty]^\perp$ (in the topology inherited by $X$) and the equation (\ref{eq:algriccX}) is satisfied for all $x,y\in \cald(A)\cap \cald(R)$.
\end{itemize}
\end{Definition}

%
\begin{Proposition}\label{pr:reldueARE}
The following facts are equivalent.
\begin{description}
\item{(i)} $P\in \call_+(H)$ is a solution to \eqref{eq:algricc};
\item{(ii)} $R=Q_\infty^{-1}P$ is a solution to \eqref{eq:algriccX} and it satisfies, in addition, $Q_\infty^{1/2} R Q_\infty^{1/2} \in \call(X)$.
\end{description}
\end{Proposition}
\begin{proof}
{\bf (i)} Assume that $P\in \call_+(H)$ solves \eqref{eq:algricc}. Then, in particular the set $\cald(A)\cap \cald(\Lambda_P)$ is dense in $H$. Setting $R=Q_\infty^{-1}P$ we see that its domain is exactly $\cald(\Lambda_P)$, which is dense in $[\ker Q_\infty]^\perp$. The fact that such $R$ satisfies \eqref{eq:algriccX}
for every $x,y\in \cald(A)\cap \cald(\Lambda_P)$ follows by simple substitution. Finally, for every $x\in X$ we have
$$\|Q_\infty^{1/2} R Q_\infty^{1/2}x\|_X =\|Q_\infty^{-1/2}P Q_\infty^{1/2}x\|_X = \|P Q_\infty^{1/2}x\|_H \le \|P\|_{\call(H)} \|Q_\infty^{1/2}x\|_H = \|P\|_{\call(H)} \|x\|_X\,.$$
{\bf (ii)} Let $R:\cald(R)\to X$ be a solution of \eqref{eq:algriccX}, having the property $Q_\infty^{1/2} R Q_\infty^{1/2} \in \call(X)$: note that, in this case, $\cald(R)$ must coincide with $H$. Thus $\cald(A)\cap \cald(R)$ is dense in H, since it contains $\cald(A_0)$. We set $P=Q_\infty R$: then $P\in \call_+(H)$ since, for every $x\in H$,
\bey
\|Px\|_H  & = & \|Q_\infty Rx\|_H = \|Q_\infty^{1/2}[Q_\infty^{1/2} R Q_\infty^{1/2}] Q_\infty^{-1/2}x\|_H = \|[Q_\infty^{1/2} R Q_\infty^{1/2}]
Q_\infty^{-1/2}x\|_X \le \\[1mm]
& \le & \|Q_\infty^{1/2} R Q_\infty^{1/2}\|_{\call(X)} \|Q_\infty^{-1/2}x\|_X = \|Q_\infty^{1/2} R Q_\infty^{1/2}\|_{\call(X)} \|x\|_H\,.
\eey
Moreover, we see immediately that $\cald(\Lambda_P)=H$. In addition, \eqref{eq:algriccX} transforms into \eqref{eq:algricc}, and it holds for every $x,y\in \cald(A)\cap \cald(R)$, i.e. it holds for every $x,y\in \cald(A)\cap H=\cald(A)\cap \cald(\Lambda_P)$, as required by Definition \ref{df:solutionAREHX}.
\end{proof}

Concerning the two AREs (\ref{eq:algricc}) and (\ref{eq:algriccX})
we have the following result.

\begin{Theorem}\label{th:maximalARE}
Let Hypothesis \ref{NC} hold true.
\begin{itemize}
  \item[(i)] The operator $R=Q_{\infty}^{-1} $ is a solution of the
Riccati equation (\ref{eq:algriccX}) in the sense of Definition \ref{df:solutionAREHX}(ii).

  \item[(ii)] The operator $P=I_H$ is a solution of the Riccati equation (\ref{eq:algricc}) in the sense of Definition \ref{df:solutionAREHX}(i).
  \item[(iii)] Assume that $BB^*$ is coercive. Then the operator $I_H$ is the maximal solution of (\ref{eq:algricc}) in the following sense: if $\hat P$ is another solution of (\ref{eq:algricc}) in the sense of Definition \ref{df:solutionAREHX}-(i), belonging to the class $\calq$ introduced in Definition \ref{hp:hplemmachiave},
      then
$$\frac12\langle \hat P x,x \rangle_H \le\frac12 \langle x,x\rangle_H =V_\infty(x)
\qquad \forall x\in H.$$
\end{itemize}
\end{Theorem}
\begin{proof} {\bf (i)}
By \cite[Proposition 3.3]{AcquistapaceGozzi17}, $Q_\infty$ solves the Lyapunov equation, i.e. we have for every $\xi\in {\cal D}(A^*)$
$$AQ_\infty \xi+ Q_\infty A^*\xi+  BB^*\xi=0.$$
This implies that, for every $\xi\in {\cal D}(A^*)$ and $\eta \in X$,
$$
\langle AQ_\infty \xi,\eta  \rangle _X +
\langle Q_\infty A^*\xi, \eta \rangle _X
 +  \langle B^*\xi, B^*\eta \rangle _U=0.
$$
When $\eta \in \cald(AQ_\infty)$ the second term above rewrites as
$\langle \xi, AQ_\infty\eta \rangle _X$.
Consequently, when $\eta \in \cald(AQ_\infty)$, the functional
$\xi \to \<AQ_\infty \xi, \eta\>_X$, well defined since $\xi \in \cald(A^*)$, can be extended to a bounded linear operator on $X$, since it is equal to $-\<\xi,AQ_\infty \eta\>_X - \<B^* \xi, B^* \eta\>_U$.
Hence, choosing $\xi \in \cald(AQ_\infty)$, we get, for
$\xi, \eta \in \cald(AQ_\infty)$, that
\begin{equation}\label{eq:Lyapgiusta}
  \langle AQ_\infty \xi,\eta  \rangle _X +
\langle \xi, AQ_\infty\eta \rangle _X
 +  \langle B^*\xi, B^*\eta \rangle _U=0.
\end{equation}
Now set $x=Q_\infty \xi$ and $y=Q_\infty \eta$. Then $x,y \in {\cal D}(A)$ and the above rewrites as
\begin{equation}\label{eq:AREdim1}
\langle  Ax,\eta  \rangle _X + \langle \xi, Ay \rangle _X +  \langle B^*\xi, B^*\eta \rangle _U=0.
\end{equation}
Observe that $\xi=Q_\infty^{-1}x+\xi_0$ and $\eta=Q_\infty^{-1}y+\eta_0$
for suitable $\xi_0,\eta_0\in \ker Q_\infty\subseteq \ker B^*$.
Hence, using the fact that $Q_\infty$ solves the Lyapunov equation in the form (\ref{eq:Lyapgiusta}),
we have, for $\xi \in \cald(AQ_\infty)$,
$$
\langle Ax, \eta_0  \rangle _X =
\langle  AQ_\infty \xi,\eta_0  \rangle _X =
-\langle  \xi,AQ_\infty\eta_0 \rangle _X
-\langle  B^*\xi,B^*\eta_0\rangle _U=0
$$
and, similarly, for $\eta\in \cald(AQ_\infty)$, $\langle \xi_0, Ay \rangle _X =0$. We then get, substituting into \eqref{eq:AREdim1} and observing that
$B^*\xi_0= B^*\eta_0=0$,
\begin{equation}\label{eq:AREQinfty}
\langle Ax, Q_\infty^{-1} y \rangle _X
+\langle  Q_\infty^{-1}x,Ay \rangle _X + \langle B^* Q_\infty^{-1}x , B^* Q_\infty^{-1}y\rangle _U=0, \quad +x,y \in
Q_\infty ({\cal D}(AQ_\infty)).
\end{equation}
The above is exactly equation (\ref{eq:algriccX}) for $R=Q_\infty^{-1}$.
To end the proof of (i),
it is enough to observe that $Q_\infty ({\cal D}(AQ_\infty))$ is dense in $[\ker Q_\infty]^\perp$ (using Remark \ref{rm:densrem} and the fact that it contains $Q_\infty(\cald(A^*))$), and moreover that
$$
Q_\infty ({\cal D}(AQ_\infty))=\cald(A)\cap \calr(Q_\infty)=\cald(A)\cap\cald(Q_\infty^{-1}).
$$
Indeed if $x \in Q_\infty ({\cal D}(AQ_\infty))$ then it must be $x=Q_\infty \xi$ with $\xi \in ({\cal D}(AQ_\infty))$, so that
$AQ_\infty \xi$ is well defined and, clearly, it coincides with $Ax$, proving that $x\in \cald(A)$. Obviously it must also be $x \in \calr(Q_\infty)$.
The converse is similar.\\[1mm]
{\bf (ii)} It is enough to observe that \eqref{eq:AREQinfty} coincides with \eqref{eq:algricc} with $P=I_H$, and that $\cald(\Lambda_{I_H})=R(Q_\infty)$.\\[1mm]
{\bf (iii)}
Let $\hat P$ be a solution of (\ref{eq:algricc}) belonging to the class $\calq$ introduced in Definition \ref{hp:hplemmachiave}. It is immediate to see that $\hat P$ is a stationary solution of (\ref{eq:RiccatiNH}) in the sense of Definition \ref{df:hplemmachiave}.
Now we apply Lemma \ref{lem:massimalitaN} and (\ref{eq:V0etVN-new}),
getting
$$
\frac12\langle \hat Px, x \rangle_H \le V^{\hat P}(t,x) \le V(t,x) \quad
x\in H, \quad t>T_0.
$$
Taking the limit as $t\to +\infty$, the result follows by Proposition \ref{pr:VlimitVt}.
\end{proof}


\begin{Remark}\label{rm:Cuu} {\em The statement of Theorem \ref{th:maximalARE} still holds if we consider the slightly more general problem where the energy functional
has the integrand $\langle Cu,u\rangle_U$ instead of $\langle u,u\rangle_U\,$, where $C\in {\mathcal L}_+(U)$ is coercive and hence invertible. Indeed it is enough to define the new control variable $v= C^{1/2}u$ and, consequently, to replace the control operator $B$ in the state equation by $BC^{-1/2}$.}
\hfill\qedo
\end{Remark}

\begin{Remark}\label{rm:EXA} {\em Theorem \ref{th:maximalARE} can be applied to a variety of cases (e.g. delay equations treated in \cite[Subsection 5.1]{AcquistapaceGozzi17} or wave equations).
Here, according to our motivating example arising in physics, we develop more deeply the analysis when
the operator $A$ is selfadjoint and commutes with $BB^*$
and, in particular, when both are diagonal. This will be done in the next section.}
\hfill\qedo
\end{Remark}

\section{The selfadjoint commuting case}

We consider the case where $A$ is selfadjoint and invertible and commutes with $BB^*$. To apply Theorem \ref{th:maximalARE} we need that $BB^*$ is coercive; hence we assume the following:
\begin{Hypothesis}
  \label{hp:comm}
$A$ is selfadjoint and invertible and commutes with $BB^*$ i.e. for every $x \in \mathcal{D}(A)$ we have $BB^*x\in \mathcal{D}(A)$ and
$ABB^*x=BB^*Ax$.
Moreover, $BB^*$ is coercive, i.e., for a suitable $\mu>0$, $\|B^*x\|_U\ge \mu \|x\|_X$ for all $x\in X$.
\end{Hypothesis}
From \cite[Proposition C.1-(v)]{AcquistapaceGozzi17}
we know that, for every $x \in X$,
$$
Q_\infty x=-\frac12A^{-1}BB^*x.
$$
This implies that ${\cal R}(Q_\infty) ={\cal D}(A)$, and, as $BB^*$ is invertible in $X$, we have $Q_\infty^{-1}x=-2(BB^*)^{-1}Ax$ for every $x \in {\cal R}(Q_\infty)$
(see again \cite[Proposition C.1-(v)]{AcquistapaceGozzi17}).
Hence the Riccati equation \eqref{eq:algricc} in $H$ (with unknown $P\in \call(H)$), becomes
\begin{equation}\label{eq:commutingX}
0=-\< Ax,Q_\infty^{-1} Py \>_X
-\< Q_\infty^{-1} Px,Ay \>_X
+2 \< APx,Q_\infty^{-1}Py \>_X.
\end{equation}
This makes sense, as for \eqref{eq:algricc}, when $x,y \in \cald(A)\cap \cald(\Lambda_P)$ (see Definition \ref{df:hplemmachiave}).
We now want to rewrite this equation using the inner products in $H$.
Observe first that in $\calr (Q_\infty)$ we have
$Q_\infty^{-1}=Q_\infty^{-1/2}Q_\infty^{-1/2}$.
Then, if $Ax$, $Ay$ and $APx$ belong to $H$, we rewrite \eqref{eq:commutingX} as
\begin{equation}\label{eq:commutingH}
0=-\< Ax,Py \>_H
-\< Px,Ay \>_H
+2 \< APx,Py \>_H.
\end{equation}
Now, recalling the definition of $A_0$ (see Notation \ref{eq:A0} and Lemma \ref{exptAHH})-(ii)), equation \eqref{eq:commutingH} can be equivalently rewritten as
\begin{equation}\label{eq:commutingA0}
0=-\langle A_0x,Py \rangle_H -\langle Px, A_0y \rangle_H + 2\langle A_0 Px,Py \rangle_H,
\end{equation}
provided that $x,y,Px,Py$ belong to ${\cal D}(A_0)$.\\
We now clarify the relationship between \eqref{eq:commutingX} and \eqref{eq:commutingA0}. First we set
\begin{equation}\label{eq:D2}
D^P:=\left\{ x \in {\cal D}(A_0): \;  Px \in {\cal D}(A_0) \right\}.
\end{equation}
%
%
Next, we provide the following definition of solution for \eqref{eq:commutingA0} (compare with Definition \ref{df:solutionAREHX}):
\begin{Definition}\label{df:solutionAREHcomm}
An operator $P\in \Lc_+ (H)$ is a solution of the ARE (\ref{eq:commutingA0})
if the set $D^P$ is dense in $H$ and the equation (\ref{eq:commutingA0})
is satisfied for every $x,y\in D^P$.
\end{Definition}
Finally, we observe that every solution of \eqref{eq:commutingX}
is also a solution of (\ref{eq:commutingA0}): indeed, if $P\in \Lc_+ (H)$, then, by definition, we have $D^P\subseteq \cald(A)\cap \cald(\Lambda_P)$. Hence, if $P\in \Lc_+(H)$ solves equation (\ref{eq:commutingX}), then, choosing in particular $x,y\in D^P$ we can turn (\ref{eq:commutingX}) into (\ref{eq:commutingA0}).

The reverse procedure is also possible: we postpone the proof at the end of the Section, since some more informations on solutions $P$ of \eqref{eq:commutingA0} are needed.\\[2mm]
We now give a preparatory result about the properties of such solutions.

\begin{Proposition}\label{pr:pre-commuting}
Assune Hypothesis \ref{hp:comm}. Then any solution $P$ of (\ref{eq:commutingA0}) satisfies
\begin{equation}\label{eq:symm}
\<A_0x,A_0Pz\>_H=\<A_0Px,A_0z\>_H\qquad \forall x,z \in D^P.
\end{equation}
\end{Proposition}
\begin{proof}
Let $P$ be a solution of \eqref{eq:commutingA0}. We observe that for all $x,y \in D^P$ we have, since $A_0$ is selfadjoint in $H$ (see Lemma \eqref{lm:QinftyA0}-(iii)),
\begin{equation}\label{eq:prima}
\<A_0 Px,y\>_H+\<PA_0 x, y\>_H = 2\<PA_0 Px,y\>_H\,.
\end{equation}
By density, this equation holds for every $x\in D^P$ and $y\in H$. Symmetrically we have also
\begin{equation}\label{eq:prima_bis}
\<x,PA_0y\>_H+\<x,A_0P y\>_H = 2\<x,PA_0 Py\>_H
\end{equation}
for every $x\in H$ and $y\in D^P$. We choose in \eqref{eq:prima} $y=PA_0z-A_0Pz$, with $z\in D^P$, and we obtain:
\bey
\lefteqn{\langle A_0 Px, PA_0z \rangle_H - \langle A_0 Px, A_0Pz \rangle_H + \langle PA_0x,PA_0z \rangle_H - \langle PA_0x,A_0Pz \rangle_H}\\
& & \qquad \qquad \qquad \qquad \qquad \qquad = 2\langle PA_0Px,PA_0z \rangle_H - 2\langle PA_0Px,A_0Pz \rangle_H\,.
\eey
We isolate on the left the symmetric terms:
\bey
\lefteqn{2\langle PA_0Px,A_0Pz \rangle_H - \langle A_0 Px, A_0Pz \rangle_H + \langle PA_0x,PA_0z \rangle_H}\\
& & \qquad \qquad \qquad = -\langle A_0 Px, PA_0z \rangle_H+\langle PA_0x,A_0Pz \rangle_H+ 2\langle PA_0Px,PA_0z \rangle_H\,.
\eey
Next, we apply \eqref{eq:prima} to the last term on the right:
\bey
\lefteqn{2\langle PA_0Px,A_0Pz \rangle_H - \langle A_0 Px, A_0Pz \rangle_H + \langle PA_0x,PA_0z \rangle_H}\\
& & = -\langle A_0 Px, PA_0z \rangle_H+\langle PA_0x,A_0Pz \rangle_H+ \langle A_0Px,PA_0z \rangle_H + \langle PA_0x,PA_0z \rangle_H\,,
\eey
which simplifies to
$$2\langle PA_0Px,A_0Pz \rangle_H -\langle A_0 Px,A_0Pz \rangle_H= \langle PA_0x,A_0Pz \rangle_H\,.$$
Applying \eqref{eq:prima_bis} to the term on the right, rewritten as $\langle A_0x,PA_0Px\rangle_H$, we obtain for every $x,z\in D^P$
\begin{equation}\label{eq:seconda} 2\langle PA_0Px,A_0Pz \rangle_H -\langle A_0 Px,A_0Pz \rangle_H -\frac12 \langle PA_0x,A_0z\rangle_H =\frac12 \langle A_0x,A_0Pz\rangle_H\,.
\end{equation}
We now restart from \eqref{eq:prima_bis}, and choose $x=PA_0z-A_0Pz$, with $z\in D^P$: acting on the left variable of the inner product, and proceeding exactly in the same way as before, we get for every $z,y\in D^P$
\begin{equation}\label{eq:terza} 2\langle A_0Pz,PA_0Py \rangle_H -\langle A_0 Pz,A_0Py \rangle_H -\frac12 \langle PA_0z,A_0y\rangle_H =\frac12 \langle A_0Pz,A_0y\rangle_H\,.
\end{equation}
Comparing equations \eqref{eq:seconda} and \eqref{eq:terza}, both written with variables $x,y$, we immediately obtain
$$\frac12 \langle A_0x,A_0Py\rangle_H = \frac12 \langle A_0Px,A_0y\rangle_H\,, \quad x,y\in D^P,$$
which is \eqref{eq:symm}.
\end{proof}
We can now prove:
\begin{Theorem}\label{th:commuting} Assume Hypothesis \ref{hp:comm}. Then any solution $P$ of (\ref{eq:commutingA0}) commutes with $A_0$, i.e. $Px\in \cald(A_0)$ for every $x\in \cald(A_0)$ and
$$A_0Px = PA_0x \qquad \forall x\in \cald(A_0).$$
In particular $D^P=\cald(A_0)$.
\end{Theorem}
\begin{proof}
We start from \eqref{eq:symm} with $w=A_0x$ and $y=A_0z$, i.e.
\begin{equation}\label{eq:symm2}
\<w,A_0PA_0^{-1}y\>_H=\<A_0PA_0^{-1}w,y\>_H\qquad \forall w,y \in A_0(D^P).
\end{equation}
Notice that $A_0(D^P)$ is the natural domain of the operator $A_0PA_0^{-1}$; which might be ({\em a priori}) not dense in $H$. Let us denote by $Z$ the closure of $\cald(A_0PA_0^{-1}) $in $H$; so we have
$$Z:=\overline{A_0(D^P)} = \overline{\cald(A_0PA_0^{-1})}.$$
Obvoiusly $Z$ is a Hilbert space with the inner product of $H$.
Equation \eqref{eq:symm2} then tells us that $A_0(D^P)\subseteq \cald((A_0PA_0^{-1})^*)$ and
\begin{equation}\label{eq:autoagg1}
(A_0PA_0^{-1})^*w = A_0PA_0^{-1}w \qquad \forall w\in A_0(D^P)=\cald
(A_0PA_0^{-1}).
\end{equation}
On the other hand, if $x\in \cald(A_0)$ and $y\in \cald(A_0PA_0^{-1})$ we may write
$$\<x,A_0PA_0^{-1}y\>_H =\<A_0^{-1}PA_0x,y\>_H\,;$$
consequently
\begin{equation}\label{eq:domdenso}
\cald(A_0) \subseteq \cald((A_0PA_0^{-1})^*)
\end{equation}
and
\begin{equation}\label{eq:autoagg2}
(A_0PA_0^{-1})^*x = A_0^{-1}PA_0x \qquad \forall x\in \cald(A_0).
\end{equation}
We now claim that $A_0PA_0^{-1}$ is selfadjoint in the space $H$, i.e.
\begin{equation}\label{eq:domains}
\cald((A_0PA_0^{-1})^*)=\cald(A_0PA_0^{-1})=A_0(D^P)
\end{equation}
is dense in $H$ and \eqref{eq:autoagg1} holds.\\
Indeed, assume that $z\in \cald((A_0PA_0^{-1})^*)$: then there is $c>0$ such that
$$|\<A_0PA_0^{-1}x,z\>_H|\le c\|x\|_H \qquad\forall x\in \cald(A_0PA_0^{-1}).$$
In particular, by \eqref{eq:autoagg1},
$$\<x,(A_0PA_0^{-1})^*z\>_H = \<A_0PA_0^{-1}x,z\>_H =\<(A_0PA_0^{-1})^*x,z\>_H \quad \forall x\in \cald(A_0PA_0^{-1}).$$
This shows that $z\in \cald(A_0PA_0^{-1})$ and $A_0PA_0^{-1}z=(A_0PA_0^{-1})^*z$. Hence
$$\cald((A_0PA_0^{-1})^*)\subseteq \cald(A_0PA_0^{-1}) \quad \textrm{and} \quad A_0PA_0^{-1}=(A_0PA_0^{-1})^* \ \textrm{on} \ \cald((A_0PA_0^{-1})^*).$$
Conversely, we know from \eqref{eq:autoagg1} that
$$\cald(A_0PA_0^{-1})=A_0(D^P)\subseteq \cald((A_0PA_0^{-1})^*) \quad \textrm{and} \quad (A_0PA_0^{-1})^*=A_0PA_0^{-1} \ \textrm{on} \ \cald(A_0PA_0^{-1}).$$
In particular, by \eqref{eq:domdenso}, $Z$ coincides with $H$, i.e. both domains in \eqref{eq:domains} are dense in $H$. This proves our claim.\\[1mm]
Take now $x\in \cald(A_0)$. As, by \eqref{eq:domdenso}, $D(A^0) \subseteq \cald(A_0PA_0^{-1})$, we have
\begin{equation}\label{carattDP}
PA_0^{-1}\in \cald(A_0) \quad \forall x\in \cald(A_0), \quad \textrm{i.e.} \quad D^P=\cald(A_0).
\end{equation}
(see \eqref{eq:D2}). Moreover, by \eqref{eq:autoagg2} and by the above claim we deduce
$$A_0^{-1}PA_0x = A_0PA_0^{-1}x \qquad \forall x\in \cald(A_0).$$
Applying $A_0^{-1}$ we have $A_0^{-2}PA_0x = PA_0^{-1}x$ for every $x\in \cald(A_0)$, or, equivalently,
$$A_0^{-2}PA_0^2z = z \qquad \forall z\in \cald(A_0^2),\qquad \textrm{i.e.}
\qquad A_0^{-2}Pw = PA_0^{-2}w \qquad \forall w\in H.$$
This means that the bounded operators $A_0^{-2}$ and $P$ commute. Now, since $A_0^{-1}$ is a non-negative operator such that $(A_0^{-1})^2=A_0^{-2}$, by a well known result (see \cite[Theorem VI.9]{Reed-Simon}), $A_0^{-1}$ must commute with every bounded operator $B$ which commutes with $A_0^{-2}$, for instance $B=P$. So
$$A_0^{-1}Pw = PA_0^{-1}w \qquad \forall w\in H,
\qquad \textrm{i.e.} \qquad Pz = A_0^{-1}PA_0z \qquad \forall z\in \cald(A_0);$$
this implies that $P(\cald(A_0)) \subseteq \cald(A_0)$ and $A_0Pz=PA_0z$ for every $z\in \cald(A_0)$. Thus $P$ commutes with $A_0$, as required.
Moreover $P(\cald(A_0)) \subseteq \cald(A_0)$ implies $\cald(A_0)\subseteq D^P$.
The reverse inclusion immediately follows from the definition of $D^P$.
\end{proof}
We are now able to characterize all solutions of the ARE \eqref{eq:commutingA0}.
\begin{Theorem}\label{th:sol=proj}
Assume Hypothesis \ref{hp:comm} and let $P\in \call_+(H)$.
Then $P$ is a solution of (\ref{eq:commutingA0}) if and only if $P$ is an orthogonal projection in $H$ and it
commutes with $A_0$. In particular the identity $I_H$ is the maximal solution among all solutions of (\ref{eq:commutingA0}).
\end{Theorem}
\begin{proof}
Let $P$ be a solution of \eqref{eq:commutingA0}: by Theorem \ref{th:commuting} we have $Px\in \cald(A_0)$ for every $x\in \cald(A_0)$ and $A_0Px=PA_0x$. Hence the ARE \eqref{eq:prima}, equivalent to \eqref{eq:commutingA0}, becomes
$$0=-2\langle PA_0x,y \rangle_H + 2\langle PA_0 Px,y \rangle_H, \quad x\in D^P, \ y\in H.$$
Since $y$ is arbitrary, using \eqref{carattDP} we get $2PA_0x = 2PA_0Px$ for every $x\in \cald(A_0)$, and successively, for all $x\in D(A_0)$,
$PA_0x-PA_0Px=0$ , $PA_0(I_H-P)x=0$, $A_0P(I_H-P)x=0$, $P(I_H-P)x=0$, $Px=P^2x$; finally, by density, $P=P^2$.\\
Assume, conversely, that $P$ is an orthogonal projection in $H$
and it commutes with $A_0$. Then
$$PA_0Pz=P^2A_0z=PA_0z=A_0Pz \quad \forall z\in \cald(A_0),$$
and consequently $P$ solves \eqref{eq:prima}.
Finally, since $I_H$ solves \eqref{eq:commutingA0}, the last statement is immediate.
\end{proof}
We conclude this Section proving the equivalence of the two forms \eqref{eq:commutingX} and \eqref{eq:commutingA0} of the ARE.

\begin{Proposition}
  \label{pr:equivcomm}
Every solution of \eqref{eq:commutingX} is also a solution of \eqref{eq:commutingA0} and vice versa.
\end{Proposition}
\begin{proof}
We have already seen that every solution of \eqref{eq:commutingX} is also a solution of \eqref{eq:commutingA0}.\\
Consider now a solution $P$ of \eqref{eq:commutingA0}. First of all, if $x,y\in D^P=\cald(A_0)$, equation \eqref{eq:commutingA0} transforms into \eqref{eq:commutingX}, so that \eqref{eq:commutingX} holds true for $x,y\in D^P$.\\
We claim that $D^P$ is dense in $\cald(A)\cap \cald(\Lambda_P)$ (see \eqref{eq:D2}) with respect to the norm $\|\cdot\|_H +\|A\cdot\|_X +\|AP\cdot \|_X$.
Indeed, for $z\in \cald(A)\cap \cald(\Lambda_P)$, recalling Lemma \ref{exptAHH}, we set
$$z_n = nR(n,A)z= nR(n,A)|_H\,z = nR(n,A_0)z.$$
Then $z_n\in \cald(A_0)=D^P$ and, as $n\to \infty$,
$$\begin{array}{c} z_n \to z \quad \textrm{in } H, \\[2mm]
A_0z_n = nA_0R(n,A_0)z = nAR(n,A)z \to Az \quad\textrm{in } X,\\[2mm]
A_0Pz_n = nA_0 PR(n,A_0)z = nAR(n,A)Pz \to APz \quad \textrm{in } X;
\end{array}$$
this proves our claim.\\
Let now $x,y\in \cald(A)\cap \cald(\Lambda_P)$; select $\{x_n\}, \{y_n\}\subseteq \cald(A_0)$ such that, as $n\to \infty$,
$$\begin{array}{l} x_n\to x\ \textrm{in }H, \quad Ax_n \to Ax \ \textrm{in }X, \quad APx_n \to APx \ \textrm{in }X,\\[2mm]
y_n\to y\ \textrm{in }H, \quad Ay_n \to Ay \ \textrm{in }X, \quad APy_n \to APy \ \textrm{in }X.\end{array}$$
As a consequence,
$$Q_\infty^{-1}Px_n = -2BB^*APx_n \to -2BB^*APx= Q_\infty^{-1}Px \ \textrm{in } X \ \textrm{as } n\to \infty,$$
and similarly $Q_\infty^{-1}Px_n \to Q_\infty^{-1}Py$ in $X$ as $n\to \infty$.
For $x_n$ and $y_n$, \eqref{eq:commutingX} holds:
$$0=-\< Ax_n,Q_\infty^{-1} Py_n \>_X
-\< Q_\infty^{-1} Px_n,Ay_n \>_X
+2 \< APx_n,Q_\infty^{-1}Py_n \>_X.$$
In all terms, by what established above, we can pass to the limit as $n\to \infty$, obtaining
$$0=-\< Ax,Q_\infty^{-1} Py \>_X
-\< Q_\infty^{-1} Px,Ay \>_X
+2 \< APx,Q_\infty^{-1}Py \>_X \qquad \forall x,y \in \cald(A) \cap \cald(\Lambda_P),$$
i.e. $P$ solves \eqref{eq:commutingX}.
\end{proof}
\begin{Remark}\label{rm:max} {\em It is easy to verify that for every solution $P$ of \eqref{eq:commutingA0} the space $D^P=\cald(A_0)$ is dense in $\cald(A)\cap H$ with respect to the norm $\|\cdot\|_H + \|A\cdot\|_X$: it suffices to repeat the argument above, i.e. to consider, for fixed $x \in \cald(A)\cap H$, the approximation $x_n = nR(n,A_0)x$, observing that $x_n\to x$ in $H$ and $A_0x_n =Ax_n \to Ax$ in $X$. Thus, $P$ belongs to the class $\calq$ introduced in Definition \ref{hp:hplemmachiave}, and consequently, by Theorem \ref{th:maximalARE}, we have $P\le I_H$. Of course, this follows as well by Theorem \ref{th:sol=proj}.}
\end{Remark}


\begin{Corollary}
Assume that $A_0$ is a diagonal operator with respect to an orthonormal complete system $\{e_n\} $ in $H$ with sequence of eigenvalues $\{\lambda_n\}\subset \,]-\infty, 0[\,$. Let $P$ be a solution of the ARE (\ref{eq:commutingA0}).
Then
\begin{itemize}
  \item[(i)] every eigenspace of $A_0$ is invariant for $P$;
  \item[(ii)] if all eigenvalues are simple, then $P$ is diagonal
  with respect to the system $\{e_n\} $, too;
  \item[(iii)] if at least one eigenspace $M$ has dimension $m\ge 2$, then
  the restriction of $P$ to $M$ needs not be diagonal: for instance, if
$m=2$ a non-diagonal $P$ on $M$ must have the following explicit form:
\begin{equation}\label{eq:exP}
    \left(
      \begin{array}{cc}
        a &  \pm \sqrt{a(1-a)}\\
        \pm \sqrt{a(1-a)} & 1-a \\
      \end{array}
    \right) \qquad \textrm{for some }\ a \in \,]0,1[\,.
\end{equation}
\end{itemize}
\end{Corollary}
\begin{proof}
To prove (i) it is enough to show that, for every eigenvalue $\lambda$ of $A_0$
and $x$ eigenvector of $A_0$ associated to $\lambda$, we have
$\lambda Px= A_0Px$. This is immediate since $A_0$ and $P$ commute.

Concerning (ii) we observe that, for every $n \in \N$
we have $A_0e_n=\lambda_n e_n$, so that $\lambda_n Pe_n=A_0Pe_n$.
Since $\lambda_n$ is simple, it is $Pe_n=ke_n$ for some $k\in \R$.
Since $P$ is a projection, it must be $k=0$ or $k=1$.

Finally (iii) can be proved with straightforward algebraic calculations,
using the fact that $M$ is invariant under $P$ and that $P$ is a projection.
\end{proof}

\begin{Remark}
{\rm Let $A$ be a diagonal operator with respect to an orthonormal complete system $\{e_n\} $ in $H$ with sequence of eigenvalues $\{\lambda_n\}\subset \,]-\infty, 0[\,$, where all $\lambda_n$ are distinct and simple. Then $BB^*$ must be diagonal, too. Indeed we have, for every $n \in \N$,
$$\sum_{k=0}^{+\infty}\<BB^*e_n,e_k\>_H\,e_k = BB^*e_n = \frac{1}{\lambda_n}BB^*Ae_n = \frac{1}{\lambda_n}ABB^*e_n = \frac{1}{\lambda_n}\sum_{k=0}^{+\infty} \lambda_k\<BB^*e_n,e_k\>_H\,e_k\,,$$
which implies
$$\<BB^*e_n,e_k\>_H\left(1-\frac{\lambda_k}{\lambda_n}\right)=0 \quad \forall k,n \in \N.$$
Since all eigenvalues are distinct, it must be $BB^*e_n=b_n e_n$ for all $ n \in \N$ for a suitable sequence $\{b_n\}\in \ell^\infty$.
This implies that $Q_\infty$ and $Q_\infty^{1/2}$ are diagonal with respect to $\{e_n\}$, too. Following \cite[Subsection 5.2]{AcquistapaceGozzi17}
we may also consider the case when $BB^*$ is unbounded and characterize the space $H$, for specific choices of $BB^*$, in terms of the domain of suitable powers of $(-A)$. In Section \ref{PHYSICS} we will consider a specific diagonal case arising in mathematical physics.}
\hfill\qedo
\end{Remark}
\section{A motivating example: from equilibrium to non-equilibrium states}\label{PHYSICS}
In this section we describe, in a simple one-dimensional case, the optimal control
problem outlined in the papers \cite{BDGJL1,BDGJL2,BDGJL3,BDGJL4,BDGJL5,BertiniGabrielliLebowitz05}. Such special case fits into the application studied e.g. in \cite{BDGJL4,BertiniGabrielliLebowitz05}, in the case of the Landau-Ginzburg model.

We consider a controlled dynamical system whose state
variable is described by a function
$\rho :\left] -\infty ,0\right]$
(the choice of the letter $\rho$ comes from the fact that in many physical models $\rho $ is a density).
The control variable is a function
$F :\left]-\infty ,0\right] \times \left[ 0,1\right] \rightarrow \mathbb{R}$ which we assume to belong to
$L^{2}\left( -\infty ,0;L^{2}(0,1) \right)$.
The state equation is formally given by
\begin{equation}
\label{eq:statePDE}
\begin{cases}
\frac{\partial \rho }{\partial t}\left(t,x\right)
=\frac{1}{2}
\frac{\partial^2 \rho }{\partial x^2}\left(t,x\right)
 +\nabla F \left( t,x\right), & t\in\,]-\infty, 0[\,,\ x\in\,]0,1[\,,\\[1mm]
\rho \left( -\infty,x \right) =\bar{\rho}(x), & x\in[0,1],\\
\rho \left( t,0\right) =\rho _{-},\qquad \rho \left( t,1\right)
=\rho _{+},& t\in\,]-\infty, 0[\,,\\
\rho \left( 0,x\right) =\rho _{0}\left( x\right), & x\in[0,1],
\end{cases}
\end{equation}
where
$\rho _{+}, \rho _{-}\in \left( 0,1\right)$, and $\bar{\rho}$ is an equilibrium state for the uncontrolled
problem. Hence $\bar{\rho}$ is the unique solution of the following system
\[
\left \{
\begin{array}{l}
v'' \left( x\right)  =0, \\
v \left( 0\right) =\rho _{-}, \\
v \left( 1\right)=\rho _{+};
\end{array}
\right .
\]
so we have $\bar \rho(x)=(\rho_+-\rho_-)x+\rho_-$.

For any datum $\rho_0\in L^2(0,1)$ we consider any control driving (in equation \eqref{eq:statePDE}) the equilibrium state $\bar\rho$ (at time $t=-\infty$) to $\rho_0$ at time $t=0$.
Then we consider the problem of minimizing, over the set of such controls, the energy functional
\begin{equation*}
J^0_\infty\left( F \right) =
\frac{1}{2}\int_{-\infty}^{0}\|F (s)\|^2_{L^2(0,1)}\, \ud s.
\end{equation*}
Given the above structure it is natural to consider the new control
$$
\nu=\nabla F\in L^2\left(-\infty,0;H^{-1}(0,1)\right)
$$
and take both the state space $X$ and the control space $U$ equal to
$H^{-1}\left(0,1\right)$.
We now rewrite (\ref{eq:statePDE}) in our abstract setting
as follows.
First we denote by $A$ the Laplace operator in the space $H^{-1}(0,1)$ with Dirichlet boundary conditions, i.e.
$$\cald\left( A\right)  =H^1_0\left( 0,1\right),
\qquad A\eta =\eta'' \quad \forall \eta\in H^1_0\left( 0,1\right).$$
%
Hence, formally, the state equation \eqref{eq:statePDE} becomes
\begin{equation}
\label{eq:stateinfdim-prima}
\left \{
\begin{array}{l}
\rho ^{\prime }(t) =A[\rho (t) -\bar{\rho}] +\nu (t), \quad t<0, \\[1mm]
\rho( -\infty) =\bar{\rho}.
\end{array}
\right.
\end{equation}
Using a standard argument
(see e.g. \cite[Appendix C]{FabbriGozziSwiech17}),
the state equation (\ref{eq:statePDE}) can be rewritten in the space $X$ and in the new variable $y(t):=\rho(t)-\bar \rho$ as
\begin{equation}
\label{eq:stateinfdim-seconda}
\left \{
\begin{array}{l}
 y^{\prime }(t) =Ay(t) +\nu (t), \quad t<0, \\[1mm]
y( -\infty) =0.
\end{array}
\right.
\end{equation}
The function
\begin{equation}
\label{eq:mildz}
y( t;-\infty,0,\nu) =\int_{-\infty }^{t}e^{(t-s)A}\nu(s)\, \ud s, \qquad t \le 0,
\end{equation}
corresponding to
$\rho (t;\nu) = \bar \rho +\int_{-\infty }^{t}e^{(t-s)A}\nu(s)\, \ud s$,
is the unique solution of (\ref{eq:stateinfdim-seconda}),
adopting Definition \ref{def:delsolstate-new} and applying Lemma \ref{lm:existencesol-new}.

The energy functional, in the new control variable $\nu$, becomes
\begin{equation*}
\bar J^0_\infty\left( \nu \right)
=\frac{1}{2}\int_{-\infty}^{0}\| A^{-1/2} \nu (s)\|^2_{L^2(0,1)}\, \ud s
=\frac{1}{2}\int_{-\infty}^{0}\|\nu (s)\|^2_{H^{-1}(0,1)}\, \ud s.
\end{equation*}
The set of admissible controls here is exactly $\mathcal{U}_{[-\infty,0]}(0,y_0)$ (see Subsection \ref{SSE:introinfhor}), which is nonempty if and only if $y_0\in H:=R(Q_\infty^{1/2})=D(A^{1/2})=L^2(0,1)$
(see e.g. \cite[Section 5.2]{AcquistapaceGozzi17}).
The value function $V_\infty$ is defined as
\begin{equation}
V_\infty\left(y_{0}\right) :=
\inf_{\nu \in \mathcal{U}_{[-\infty,0]}(0,y_0)}
\bar J^0_\infty\left( \nu \right).  \label{eq:defvf}
\end{equation}
%
%
%
%
%
%
Now, recalling that $X=U=H^{-1}(0,1)$ and setting
$B=I_{H^{-1}(0,1)}\in \call(U,X)$,
this problem belongs to the class of
the minimum energy problems studied in this paper.
We know, from Proposition \ref{pr:VlimitVt}, that the value
function is given by
$$
V_{\infty}(y_0) =
\frac 12\| y_0\|^2_{L^2(0,1)}.
$$
We can now apply Theorem \ref{th:maximalARE}, obtaining that:
\begin{itemize}
  \item the identity in $L^2$, $I_{L^2(0,1)}$, solves the ARE \eqref{eq:algricc} where we replace $B$ and $B^*$ by $I_{H^{-1}(0,1)}$;
  \item the operator $Q_\infty^{-1}=2A$ solves the ARE \eqref{eq:algriccX} where we replace $B^*$ by $I_{H^{-1}(0,1)}$;
  \item $I_{L^{2}(0,1)}$ is the maximal solution of the ARE \eqref{eq:algricc} among those in the class $\mathcal{Q}$ introduced in Definition \ref{hp:hplemmachiave}.
\end{itemize}

Moreover, here Hypothesis \ref{hp:comm} holds; hence we can apply
Theorem \ref{th:sol=proj}. Then, noting that $A_0$ is the Laplace operator with Dirichlet boundary conditions in the space $H=L^2(0,1)$, whose domain is $H^2(0,1) \cap H^1_0(0,1)$, we obtaing that:
\begin{itemize}
  \item the identity in $L^2$, $I_{L^2(0,1)}$, is a solution of the two (equivalent) AREs \eqref{eq:commutingX} and \eqref{eq:commutingA0};
  \item the set of all solutions of \eqref{eq:commutingX} and \eqref{eq:commutingA0} consists of all orthogonal projections $P$ which commute with $A_0$, i.e. all projections whose image is generated by a subset of the eigenvectors of $A_0$;
  \item $I_{L^{2}(0,1)}$ is the maximal solution among all solutions of \eqref{eq:commutingX} and \eqref{eq:commutingA0}.
\end{itemize}

\section*{Appendix}
\appendix

\section{Minimum Energy with finite horizon}
\label{sub:finitehorizon}

This part of the Appendix is devoted to recall
the formulation of the finite horizon minimum energy problem studied in \cite{AcquistapaceGozzi17} (briefly described at the beginning of Subsection \ref{SSE:METHOD}) and
to provide some related results which are useful in treating the infinite horizon problem (\ref{storinf})--(\ref{fzorinf}). Throughout this section we will assume that Hypothesis \ref{hp:main} holds without repeating it.

\subsection{General formulation of the problem}
\label{SE:MAINPROBLEM}

We take the Hilbert spaces $X$ (state space) and $U$ (control space), as well as the operators $A$ and $B$, as in Hypothesis \ref{hp:main}.
Given a time interval $[s,t]\subset \R$, an initial state $z\in X$ and a control $u\in L^2(s,t;U)$ we consider the
state equation \eqref{eq:state-fin-new}, which we rewrite here:
\begin{equation}
\label{eq:state-new}
\left\{
\begin{array}{l}
y'(r)=Ay(r)+Bu(r), \quad r\in \,]s,t], \\[1mm]
y(s) = z.
\end{array}
\right.
\end{equation}
Denote by $y(\cdot; s,z,u)$ the mild solution of
(\ref{eq:state-new}) (see Proposition \ref{prop:solcontinua}):
\begin{equation}
\label{eq:mild-state-new} y(r;s,z,u) := e^{(r-s)A}z + \int_s^{r}
e^{(r-\tau)A} B u(\tau)\, \ud \tau, \qquad r\in [s,t].
\end{equation}
We define the class of controls $u(\cdot)$ bringing the state $y(\cdot)$ from a
fixed $z\in X$ at time $s$ to a given target $x\in X$ at time $t$:
\begin{equation}
\label{eq:contr-x-x0} {\cal U}_{[s,t]}(z,x) \nd \left\{ u\in
L^2(s,t;U) \; : \; y(t;s,z,u)=x \right \}.
\end{equation}
Consider the quadratic functional (the energy)
\begin{equation}\label{eq:energyfunctional}
J_{[s,t]}(u) = \frac12 \int_s^t \|u(r)\|_U^2\, \ud r.
\end{equation}
The minimum energy problem at $(s,t;z,x)$ is the problem of
minimizing the functional $J_{[s,t]}(u)$ over all $u \in {\cal
U}_{[s,t]}(z,x)$. The value function of this control problem (the
{\em minimum energy}) is
\begin{equation} \label{eq:valuefunction-new} V_1(s,t;z,x)\nd
\inf_{u\in {\cal U}_{[s,t]}(z,x)} J_{[s,t]}(u).
\end{equation}
with the agreement that the infimum over the emptyset is $+\infty$.
Similarly to what we did in Proposition \ref{pr:Ubarnonempty}, given any $z\in X$ we define the {\em reachable set} in the interval
$[s,t]$, starting from $z$, as
\begin{equation}\label{eq:ptiragg-new}
{\mathbf R}_{[s,t]}^z:= \left\{ x \in X:\ {\cal U}_{[s,t]}(z,x) \neq \emptyset \right\}.
\end{equation}
Defining the operator
\begin{equation}\label{eq:newLst}
{\cal L}_{s,t}:L^2(s,t;U) \to X, \qquad {\cal L}_{s,t}u=\int_s^{t}
e^{(t-\tau)A} B u(\tau)\, \ud \tau,
\end{equation}
it is clear that
\begin{equation}\label{eq:ptiragg-newbis}
{\mathbf R}_{[s,t]}^z:= e^{(t-s)A}z+{\cal L}_{s,t}\left(L^2(s,t;U)\right).
\end{equation}
The use of \cite[Proposition 2.6]{AcquistapaceGozzi17} allows to reduce the number of variables from 4 to 2. In particular
\begin{equation}\label{eq:V1V1}
V_1(s,t;z,x)=V_1(s-t,0;0,x-e^{(t-s)A}z) =V_1(0,t-s;0,x-e^{(t-s)A}z);
\end{equation}

Hence from now on we set, for simplicity of notation,
\begin{equation}\label{eq:defV0-new}
V(t,x) := V_1(-t,0;0,x) = \inf_{u\in {\cal U}_{[-t,0]}(0,x)}
J_{[-t,0]}(u) \qquad t\in \, ]0,+\infty[,\quad x\in X.
\end{equation}

\subsection{The space $H$ and its properties}
\label{SSE:SPACEH}

In this subsection we provide some useful properties of the space $H$ introduced in Subsection \ref{SSE:METHOD} (see \eqref{eq:defH}-\eqref{eq:innerproductH}).
First recall that
\begin{equation}\label{spH}
H=\calr(Q_\infty^{1/2}) \qquad \hbox{and} \qquad
\langle x, y\rangle_H = \langle Q_\infty^{-1/2}x,Q_\infty^{-1/2}y \rangle_X \,,\quad x,y\in H,
\end{equation}
and that (with, in general, proper inclusion)
$$H \subseteq \overline{{\cal R}(Q_\infty^{1/2})}= [\ker Q_\infty^{1/2}]^\perp=[\ker Q_\infty]^\perp.$$
Next Lemmas \ref{HH} and \ref{dens} are exactly Lemmas
4.2 and 4.3 of \cite{AcquistapaceGozzi17}.
%
\begin{Lemma}\label{HH}
\begin{itemize}
  \item[]
  \item[(i)] The space $H$ is a Hilbert space continuously embedded into $X$.

  \item[(ii)] The space ${\cal R}(Q_\infty)$ is dense in $H$.

  \item[(iii)] The operator $Q_\infty^{-1/2}$ is an isometric isomorphism from $H$ to $[\ker Q_\infty^{1/2}]^\perp$, and in particular
\begin{equation}\label{qinfty2}
\|Q_\infty^{-1/2}x\|_X = \|x\|_H \qquad \forall x\in H.
\end{equation}
  \item[(iv)]
We have $Q_\infty^{1/2}\in \Lc(H)$ and
$$
\|Q_\infty^{1/2}\|_{\Lc(X)}=\|Q_\infty^{1/2}\|_{\Lc(H)}.
$$
  \item[(v)] For every $F \in \Lc(X)$ such that ${\cal R}(F) \subseteq H$ we have $Q_\infty^{-1/2}F \in \Lc (X)$, so that $F\in \Lc(X,H)$.
\end{itemize}
\end{Lemma}

\begin{Lemma}\label{dens} For $0<t\le +\infty$ let $Q_t$ be the operator defined by (\ref{eq:contropintro}). Then, for every $t\in [T_0, + \infty]$, the space $Q_t({\cal D}(A^*))$ is dense in $H$ and contained in ${\cal D}(A)$.
In particular ${\cal D}(A)\cap H$ is dense in $H$.
\end{Lemma}

\begin{Remark}\label{rm:densrem}
{\rm
The above Lemma immediately implies that, for every $t\in [T_0, + \infty]$, $Q_t({\cal D}(A^*))$ is dense in $[\ker Q_\infty]^\perp$ with the topology inherited by $X$, since the inclusion of $H$ into $[\ker Q_\infty]^\perp$ is continuous.}
\end{Remark}

Now we state and prove three very useful lemmas.

\begin{Lemma}\label{exptAHH} Assume Hypothesis \ref{NC}. Then we have the following:
\begin{itemize}
\item[(i)] For every $z\in H$ and $r\ge 0$ we have $e^{rA}z\in H$; moreover the semigroup $e^{tA}|_H$ is strongly continuous in $H$. In particular, for each $T>0$ there exists $c_T>0$ such that
$$\|e^{rA}z\|_H \le c_T \|z\|_H \qquad \forall z\in H, \quad \forall r\in [0,T].$$
\item[(ii)] For every $\lambda \in \rho(A)$ we have $\lambda\in \rho(A_0)$ and
$R(\lambda,A_0)=R(\lambda,A)|_H$.
\item[(iii)] The generator $A_0$ of the semigroup $e^{tA}|_H$ is given by
\begin{equation}\label{eq:DomA0}
\left\{ \begin{array}{l}
{\cal D}(A_0)=\left\{x\in {\cal D}(A) \cap H \, :\; Ax \in H\right\} \\[2mm]
A_0x = Ax \quad \forall x\in {\cal D}(A_0).\end{array} \right.
\end{equation}
\end{itemize}
\end{Lemma}
We denote the semigroup $e^{tA}|_H$ by $e^{tA_0}$ (see Notation \ref{eq:A0}): thus $e^{tA_0}$ is a strongly continuous semigroup in $H$.
\begin{proof} {\bf (i)}
Fix any $z\in H$ and $t>T_0$: then, by Hypothesis \ref{NC}, $z\in {\cal R}(Q_\infty^{1/2})={\cal R}(Q_t^{1/2})={\cal R}({\cal L}_{-t,0})$ (see \eqref{eq:ptiragg-Q12}), i.e. there exists $u\in L^2(0,r;U)$ such that
$$z={\cal L}_{-t,0}(u) = \int_{-t}^0 e^{-\sigma A}\,Bu(\sigma)\,\ud \sigma.$$
Hence, for every $r>0$,
$$e^{rA}z = \int_{-t}^0 e^{(r-\sigma)A}\,Bu(\sigma)\,d\sigma = \int_{-t}^{r} e^{(r-\sigma)A}\,B\overline{u}(\sigma)\,\ud \sigma,$$
where
$$\overline{u}(s) = \left\{ \begin{array}{ll} u(s) & \textrm{if } s\in [-t,0] \\[2mm] 0 & \textrm{if } s\in [0,r].\end{array} \right.$$
Setting $r-\sigma=-s$ and $v(s)=\overline{u}(s+r)$, it follows that
$$e^{rA}z = \int_{-t-r}^0 e^{-sA}\,B\overline{u}(r+s)\,ds = {\cal L}_{-t-r,0}(v) \in {\cal R}({\cal L}_{-t-r,0}) = {\cal R}(Q_{t+r}^{1/2}) = {\cal R}(Q_{\infty}^{1/2}) = H.$$
Let us now prove that the restriction of $e^{rA}$ to $H$ has closed graph in $H$: if $z,\{z_n\}\subset H$ and $z_n\to z$ in $H$, $e^{rA}z_n \to w\in H$ in $H$, then, since $H$ is continuously embedded into $X$,
$$z_n \to z \textrm{ in } X, \qquad e^{rA}z_n \to w \textrm{ in } X;$$
but $e^{rA}\in {\cal L}(X)$, so that $w=e^{rA}z$. Thus $e^{rA}z_n \to e^{rA}z$ in $H$, and it follows that $e^{rA}\in {\cal L}(H)$. \\
Now fix $x\in H$ and consider for $t>0$ the quantity $e^{tA}x-x$. We have
$$\|e^{tA}x-x\|_H = \sup_{\|y\|_H=1} \langle e^{tA}x-x,y\rangle_H\,;$$
thus, for every $\varepsilon \in \,]0,1[\,$ there exists $y_\varepsilon\in H$ with $\|y_\varepsilon\|_H=1$ such that		
$$\|e^{tA}x-x\|_H < \varepsilon + \langle e^{tA}x-x,y_\varepsilon\rangle_H;$$
then, using Lemma \ref{dens} and choosing $z_\varepsilon \in {\cal R}(Q_\infty)$ such that $\|z_\varepsilon-y_\varepsilon\|_H<\varepsilon$, we obtain
\bey
\|e^{tA}x-x\|_H & < & \varepsilon + \langle e^{tA}x-x,y_\varepsilon-z_\varepsilon \rangle_H + \langle e^{tA}x-x,z_\varepsilon\rangle_H \le\\
& \le & \varepsilon + \|e^{tA}x-x\|_H \,\|y_\varepsilon-z_\varepsilon\|_H + \langle e^{tA}x-x,Q_\infty^{-1}z_\varepsilon\rangle_X \le \\
& \le & \varepsilon + \|e^{tA}x-x\|_H\,\varepsilon + \|e^{tA}x-x\|_X \,\|Q_\infty^{-1}z_\varepsilon\|_X\,.
\eey
Hence
$$(1-\varepsilon)\|e^{tA}x-x\|_H < \varepsilon+ \|e^{tA}x-x\|_X \,\|Q_\infty^{-1}z_\varepsilon\|_X\,,$$
and letting $t\to 0^+$ we get
$$\limsup_{t\to 0^+} \|e^{tA}x-x\|_H \le \frac{\varepsilon}{1-\varepsilon} + 0.$$
The arbitrariness of $\varepsilon$ leads to the conclusion.\\[1mm]
{\bf (ii)}
This is an immediate consequence of point (i) and of the well known resolvent formula
$R(\lambda,A)=\int_0^{+\infty}e^{-\lambda t}e^{tA}\,\ud t$.\\[1mm]
{\bf (iii)}
If $z\in {\cal D}(A_0)$ then it must be, by definition, $z \in {\cal D}(A)$, $z \in H$, $Az \in H$; hence ${\cal D}(A_0)\subseteq\left\{x\in {\cal D}(A) \cap H \, :\; Ax \in H\right\}$.
To prove the converse we first observe that, using (ii),
we get, for $n\in \N-\{0\}$ and $h\in H$,
$$
nAR(n,A)x=nx-n^2R(n,A)x=nx-n^2R(n,A_0)x=nA_0R(n,A_0)x
$$
Now assume that $z\in {\cal D}(A)\cap H$ with $Az\in H$.
To prove that $z\in {\cal D}(A_0)$ it is enough to show that
$nA_0R(n,A_0)z$ converges to some element $y$ of $H$ when $n\to + \infty$. In this case such element is $A_0z$.
To do this we observe that, by the above remarks for the resolvents and by the assumptions on $z$, we get
$$
nA_0R(n,A_0)z=nAR(n,A)z=nR(n,A)Az
=nR(n,A_0)Az
$$
The latter, by the properties of resolvents, converges in $H$ to $Az$ as $n \to +\infty$, since $Az \in H$. This shows that
$z\in {\cal D}(A_0)$ and $A_0z=Az$.
%
\end{proof}

\begin{Lemma}\label{lm:QinftyA0} Assume Hypothesis \ref{NC}.
Then we have the following.
\begin{itemize}
  \item[(i)] $Q_\infty (H)$ is dense in $H$.
  \item[(ii)] $Q_\infty ({\cal D}(A_0^*))$ is dense in $H$.
  \item[(iii)] Let $A$ be selfadjoint and commuting with $BB^*$.
  Then $Q_\infty ({\cal D}(A_0^*))\subseteq {\cal D}(A_0)$.
  Moreover $A_0^*$ is selfadjoint in $H$.
\end{itemize}
\end{Lemma}
\begin{proof} {\bf(i)}
Since $\ker Q_\infty^{1/2}=\ker Q_\infty$,
we have $\overline{{\cal R}( Q_\infty^{1/2})}=
\overline{{\cal R}( Q_\infty)}$.
Fix $x \in H$ and set $z:=Q_\infty^{-1/2}x\in \overline{{\cal R}( Q_\infty^{1/2})}$.
Then there exists $\{w_n\}\subset X$ such that, defining $z_n:=Q_\infty w_n\in {\cal R}( Q_\infty)$, we have $z_n \to z$ in $X$. Set
$$
x_n:=Q_\infty^{1/2} z_n=Q_\infty^{1/2}Q_\infty w_n=Q_\infty Q_\infty^{1/2}w_n.
$$
Clearly $x_n\in Q_\infty(H)$.
Moreover
$$
\|x_n - x\|_H=\|Q_\infty^{1/2} z_n - x\|_H=\|z_n -z\|_X \to 0 \quad \hbox{as $n \to +\infty$,}
$$
which proves the claim.\\[2mm]
[{\bf(ii)}
Fix $x \in H$. By part (i) there exists $\{x_n\}\subset Q_\infty (H)$
such that $x_n \to x$ in $H$.
We must have $x_n:=Q_\infty z_n$, with $z_n \in H$. Since
${\cal D}(A^*_0)$ is dense in $H$, then, for every $n \in \N_+$
there exists $w_n \in {\cal D}(A^*_0)$ such that
$\|z_n-w_n\|_H< 1/n$.
Consequently, setting $y_n:=Q_\infty w_n$, we have,
using Lemma \ref{HH}-(iv),
$$
\|y_n-x\|_H\le\|Q_\infty (w_n - z_n)\|_H+ \|x_n-x\|_H \le
\|Q_\infty\|_{\Lc(H)}\frac1n+\|x_n-x\|_H\, .
$$
This proves the claim.\\[2mm]
{\bf(iii)}
Let $A$ be selfadjoint and commuting with $BB^*$.
Observe first that ${\cal D}(A_0^*)\subseteq {\cal D}(A^*)={\cal D}(A)$.
Indeed, when $x \in {\cal D}(A_0^*)$, the linear map $y\to \<x,A_0y\>_H$
is bounded in $H$.
Using such boundedness and the fact that $A$ and $Q_\infty$ commute
(see \cite[Proposition C.1-(v)]{AcquistapaceGozzi17})), we get, for every $y \in {\cal D}(A)$,
$$
\<x,Ay\>_X=\<x,Q_\infty Ay\>_H=\<x,AQ_\infty y\>_H=
\<x,A_0Q_\infty y\>_H\le C\|Q_\infty y\|_H \le C'\| y\|_X
$$
which implies $x \in {\cal D}(A^*)={\cal D}(A)$.
\\
Now, let $x\in Q_\infty ({\cal D}(A_0^*))$ (which is contained in ${\cal D}(A)$ by Lemma \ref{dens},
since ${\cal D}(A_0^*)\subseteq {\cal D}(A^*)$) and let $z\in {\cal D}(A_0^*)$ be such that
$x=Q_\infty z$. Using again the fact that $A$ and $Q_\infty$ commute, we get
$Ax=AQ_\infty z =Q_\infty Az\in H$. Hence, by definition of $A_0$, we deduce that
$x \in  {\cal D}(A_0)$ and $A_0x=Ax$.\\
Now we prove that $A_0$ is selfadjoint in $H$.
Let $x \in {\cal D}(A_0)$ and $y \in Q_\infty ({\cal D}(A_0^*))$.
Then for some $z \in {\cal D}(A_0^*)$ we have $y=Q_\infty z$
and $Q_\infty^{-1} y=z+z_0$, where $z_0\in \ker Q_\infty$.
Hence it must be $\<Ax,z_0\>_X=0$, since
$Ax=A_0x\in H\subseteq [\ker Q_\infty]^\perp$.
Using this fact, we get
\bey
\<A_0x,y\>_H & = & \<Ax,y\>_H = \<Ax,Q_\infty^{-1} y\>_X =\<Ax,z\>_X = \<x,Az\>_X =\\
& = & \<x,Q_\infty A z\>_H = \<x,AQ_\infty z\>_H=\<x,Ay\>_H=\<x,A_0y\>_H\,,
\eey
where in the last step we used the inclusion ${\cal D}(A_0^*)\subseteq {\cal D}(A)$,
the fact that $Q_\infty$ and $A$ commute, and the inclusion $Q_\infty ({\cal D}(A_0^*))\subseteq {\cal D}(A_0)$.
This implies that, for every $x\in {\cal D}(A_0)$, the linear map
$y \to \<x,A_0 y\>_H$ is defined on $Q_\infty ({\cal D}(A_0^*))$ (which is dense in $H$)
and is bounded in $H$. This implies that $x \in {\cal D}(A_0^*)$ and $A_0^*x=A_0x$.
Hence $A_0^*$ extends $A_0$. Since both $A_0$ and $A_0^*$ generate a semigroup, we can choose $\lambda>0$ such that $\lambda \in \rho(A_0)\cap \rho(A_0^*)$.
For such $\lambda$ we now prove that $R (\lambda,A_0^*)=R (\lambda,A_0)$,
which immediately implies that ${\cal D}(A_0)={\cal D}(A_0^*)$. Indeed for $z \in H$ we have
$$
z=(\lambda-A_0)R (\lambda,A_0)z=(\lambda-A_0^*)R (\lambda,A_0)z,
$$
where in the last equality we used that ${\cal D}(A_0)\subseteq {\cal D}(A_0^*)$ and that $A_0^*x=A_0x$ for all $x \in {\cal D}(A_0)$. Applying $R (\lambda,A_0^*)$ to both sides we get the claim.
\end{proof}

\begin{Lemma}\label{Qe}
Assume Hypothesis \ref{NC}.
\begin{description}
\item[(i)] For $0\le s\le T_0$ we have $Q_\infty^{-1/2} e^{sA}\in \Lc(H,X)$,
and there exists $C_1(T_0)>0$ such that
$$\|Q_\infty^{-1/2} e^{sA}x\|_X \le C_1({T_0}) \|x\|_H \qquad \forall x\in H, \quad \forall s\in [0,T_0],$$
and $(Q_\infty^{-1/2}e^{sA})^*=e^{sA_0^*} Q_\infty^{-1/2}\in \Lc(X,H)$, with
$$\|(Q_\infty^{-1/2}e^{sA})^*\|_{\Lc(X,H)}\le C_1({T_0})\qquad \forall s\in [0,T_0].$$
\item[(ii)] For $s>T_0$ we have $Q_\infty^{-1/2} e^{sA}\in \Lc(X)$, with
$$\|Q_\infty^{-1/2} e^{sA}x\|_X \le C_1({T_0}) Me^{-\omega(s-T_0)}\|x\|_X \qquad \forall x\in X, \quad \forall s> T_0,$$
and
$$\|(Q_\infty^{-1/2}e^{sA})^*\|_{\Lc(X)} \le C_1({T_0}) Me^{-\omega(s-T_0)}\qquad \forall s> T_0.$$
\item[(iii)] For $s\ge 0$, $x \in X$  we have
$$e^{sA_0^*}Q_\infty x= Q_\infty e^{sA^*}x.$$
\item[(iv)] For $x\in {\cal D}(A^*)$ we have $Q_\infty x\in {\cal D}(A_0^*)$.
Moreover, for every $s\ge 0$ we have
$$A_0^* e^{sA_0^*}Q_\infty x=Q_\infty A^* e^{tA^*} x.$$
\end{description}
\end{Lemma}

\begin{proof} {\bf (i)}
We have by Lemma \ref{exptAHH}-(i)
$$\|Q_\infty^{-1/2} e^{sA}x\|_X =\|e^{sA}x\|_H \le c_{T_0} \|x\|_H \qquad \forall x\in H, \quad \forall s\in [0,T_0];$$
moreover (identifying here $X$ and $H$ with their duals), $(Q_\infty^{-1/2} e^{sA})^*\in \Lc(X,H)$ and, for all $z\in H$ and $x\in X$, we have
$$\langle (Q_\infty^{-1/2}e^{sA})^*x,z\rangle_{H}= \langle x,Q_\infty^{-1/2}e^{sA}z\rangle_X =\langle Q_\infty^{1/2}x,e^{sA_0}z\rangle_H = \langle e^{sA_0^*}Q_\infty^{1/2}x,z\rangle_H\,.$$
This shows that
$$(Q_\infty^{-1/2} e^{sA})^* =e^{sA_0^*} Q_\infty^{1/2} \in \Lc(X,H),$$
with
$$\|(Q_\infty^{-1/2} e^{sA})^*\|_{\Lc(X,H)} = \|e^{sA_0} Q_\infty^{1/2}\|_{\Lc(X,H   )}= \|Q_\infty^{-1/2} e^{sA}\|_{\Lc(H,X)} \le c_{T_0} \quad \forall s\in [0,T_0].$$
{\bf (ii)}
By Hypothesis \ref{NC} we have $Q_\infty^{-1/2} e^{sA}\in \Lc(X)$, and by (i) we get
\bey
\|Q_\infty^{-1/2} e^{sA}\|_{\Lc(X)} & = & \|Q_\infty^{-1/2} e^{T_0A}e^{(s-T_0)A}\|_{\Lc(X)}\\
& \le & \|Q_\infty^{-1/2} e^{T_0A}\|_{\Lc(X)} Me^{-\omega(s-T_0)}\le C_1({T_0})Me^{-\omega(s-T_0)} \quad \forall s>T_0.
\eey
The claim easily follows.\\[2mm]
{\bf (iii)} For $s\ge 0$, $x \in X$, $z \in H$, we have
$$\<e^{sA_0^*}Q_\infty x,z\>_H =\<Q_\infty x,e^{sA_0}z\>_H=\<x,e^{sA}z\>_X
=\<e^{sA^*}x,z\>_X=\<Q_\infty e^{sA^*}x,z\>_H\,,$$
which proves the claim.\\[2mm]
\item[(iv)] Let $x\in {\cal D}(A^*)$ and $s\ge 0$.
For $h>0$ we have, using the point (iii) above,
$$
\frac{e^{(s+h)A_0^*}-e^{sA_0^*}}{h}Q_\infty x
=Q_\infty \frac{e^{(s+h)A^*}-e^{sA^*}}{h} x
$$
Letting $h\to 0$, the claim follows.
\end{proof}


\begin{thebibliography}{99}

\bibitem{AcquistapaceGozzi17} P. Acquistapace, and F. Gozzi.
{\sl Minimum energy for linear systems with finite horizon:
a non-standard Riccati equation.} Math. Control Signals Syst. {\bf 29}, 19 (2017).

\bibitem{BarucciGozzi01} E. Barucci, and F. Gozzi.
{\sl Technology Adoption and Accumulation in a Vintage
Capital Model.} J. of Economics, {\bf 74}, N.1, 1--38 (2001).

\bibitem{BDDM07}
A. Bensoussan, G. Da Prato, M.C. Delfour, and S.K. Mitter.
\textit{Representation and control of Infinite dimensional system.
Second edition}, Birkh\"auser, Boston 2007.

%


\bibitem{BDGJL1}
L. Bertini, A. De Sole, D. Gabrielli, G. Jona--Lasinio, and C. Landim.
{\sl Fluctuations in stationary nonequilibrium states of irreversible
processes.} Phys. Rev. Lett. {\bf 87}, 040601 (2001).


\bibitem{BDGJL2}
L. Bertini, A. De Sole, D. Gabrielli, G. Jona--Lasinio, and C. Landim.
{\sl Macroscopic fluctuation theory for stationary non equilibrium state.} J. Statist. Phys. {\bf 107}, 635--675 (2002).

\bibitem{BDGJL3}
L. Bertini, A. De Sole, D. Gabrielli, G. Jona--Lasinio, and C. Landim.
{\sl Large deviations for the boundary driven simple exclusion process.}
Math. Phys. Anal. Geom. {\bf 6}, 231--267 (2003).

\bibitem{BDGJL4}
L. Bertini, A. De Sole, D. Gabrielli, G. Jona--Lasinio, and C. Landim.
{\sl Minimum Dissipation Principle in Stationary Non-Equilibrium States.}
J. Statist. Phys. {\bf 116}, 831--841 (2004).


\bibitem{BDGJL5}
L. Bertini, A. De Sole, D. Gabrielli, G. Jona--Lasinio, and C. Landim.
{\sl Action Functional and Quasi-Potential for the
Burgers Equation in a Bounded Interval.} Comm. Pure. Appl. Math. {\bf 64}, 649--696 (2011).


\bibitem{BertiniGabrielliLebowitz05}
L. Bertini, D. Gabrielli, and J.L. Lebowitz.
{\sl Large deviations for a stochastic model of heat flow.}
J. Statist. Phys. {\bf 121}, 843--885 (2005).

\bibitem{Carja93}
0. Carja. {\sl The minimal time function in infinite dimensions.} SIAM J. Contr. Optimiz. {\bf 31} (5), 1103-1114 (1993).

\bibitem{CurtainPritchardbook}
R. Curtain, and A. J. Pritchard. \textit{Infinite Dimensional Linear Systems Theory.} Springer 1978.


\bibitem{DaPratoZabczyk92}  G. Da Prato, and  J.  Zabczyk.
\textit{Stochastic Equations in Infinite Dimension.} Springer 1992.

\bibitem{DaPratoPritchardZabczyk91}
G. Da Prato, A. J. Pritchard, and J. Zabczyk. {\sl Null controllability
with vanishing energy.} SIAM J. Control Optim. {\bf 29} (1),
209--221 (1991).



\bibitem{Emir89}
Z. Emirsajlow. {\sl A feedback for an infinite-dimensional
linear-quadratic control problem with a fixed terminal state.} IMA J.
Math. Control Inf. {\bf 6} (1), 97-117 (1989).

\bibitem{Emir95}
Z. Emirsajlow, and SD. Townley. {\sl Uncertain systems and minimum
energy control.} J. Appl. Math. Comput. Sci. {\bf 5} (3),
533--545 (1995).

\bibitem{EngelNagelbook} K-J. Engel, and R. Nagel.
\textit{One Parameter Semigroups for Linear Evolution Equations.}
Springer 1999.

\bibitem{FabbriGozziSwiech17} G. Fabbri, F. Gozzi, and A. Swiech.
\textit{Stochastic optimal control in infinite dimension.} Springer 2017.


\bibitem{FengKurtzbook} J. Feng, and T. Kurtz.
\textit{Large Deviations for Stochastic Processes.}
Mathematical Surveys and Monographs, AMS 2006.


\bibitem{Goldstein85}
J.A. Goldstein. \textit{Semigroups of linear operators and applications.}
Oxford Mathematical Monographs, 1985.


\bibitem{GozziLoreti99}
F. Gozzi, and P. Loreti. {\sl Regularity of the minimum time function and
minimum energy problems: The linear case.} SIAM J. Control
Optim. {\bf 37} (4), 1195--1221 (1999).

\bibitem{LancasterARE95} P. Lancaster, and L. Rodman. \textit{Algebraic Riccati Equations.} Oxford Science Publications, Clarendon Press, Oxford, 1995.


\bibitem{LTbook1}
I. Lasiecka, and R. Triggiani.
\textit{Control Theory for Partial Differential Equations: Continuous and Approximation Theories. Part 1.} Cambridge University Press 2000.


\bibitem{LTbook2}
I. Lasiecka, and R. Triggiani.
\textit{Control Theory for Partial Differential Equations: Continuous and Approximation Theories. Part 2.} Cambridge University Press 2000.


\bibitem{Lunardi95} A. Lunardi. \textit{Analytic semigroups and optimal regularity in parabolic problems.} Birkh\"auser Verlag, Basel, 1995.

\bibitem{Moore81}
B. C. Moore. {\sl Principal component analysis in linear systems:
controllability, observability, and model reduction.} IEEE Trans.
Automat. Control \textbf{26} (1), 17--32 (1981).

\bibitem{OlbrotPandolfi88}
A. W. Olbrot, and L. Pandolfi. {\sl Null controllability of a class of functional-differential systems.} Internat. J. Control {\bf 47} (1), 193--208 (1988).

\bibitem{Pazy83}
A. Pazy. \textit{Semigroups of linear operators and applications to
partial differential equations.} Springer Verlag, New-York 1983.

\bibitem{PriolaZabczyk03}
E. Priola, and J. Zabczyk. {\sl Null controllability with vanishing
energy.} SIAM J. Control and Optim. {\bf 42} (6), 1013-1032, 2003.

\bibitem{Reed-Simon}
M. Reed, B. Simon. \textit{Functional Analysis.} Academic Press, London 1980.

\bibitem{Scherpen93} J.M.A. Scherpen. {\sl Balancing for Nonlinear Systems.} Syst. Contr. Lett. {\bf 21} (2), 143-153 (1993).

\bibitem{Triebel78}
H. Triebel. \textit{Interpolation theory, function spaces,
differential operators.} North-Holland Publishing Co., Amsterdam 1978.

\bibitem{Zabczyk92}
J. Zabczyk. \textit{Mathematical control theory: an introduction}.
Birkh\"auser Verlag, Boston 1995.





\end{thebibliography}
\end{document}